\documentclass[12pt]{article}
\RequirePackage[OT1]{fontenc}
\RequirePackage{mathrsfs,amsfonts,amsthm,amsmath,amssymb,subcaption,graphicx,epstopdf,enumerate,lmodern}
\RequirePackage[round,colon,authoryear]{natbib}
\RequirePackage[colorlinks,citecolor=blue,urlcolor=blue]{hyperref}
\usepackage{bm,color,fancyvrb,xcolor,mathtools}
\usepackage{amscd}
\usepackage{bbm}

\def\eps{\varepsilon}
\font\tencmmib=cmmib10 \skewchar\tencmmib '60
\newfam\cmmibfam
\textfont\cmmibfam=\tencmmib

\def\lessim{\ \lower4pt\hbox{$
		\buildrel{\displaystyle <}\over\sim$}\ }
\def\gessim{\ \lower4pt\hbox{$\buildrel{\displaystyle >}
		\over\sim$}\ }

\usepackage{algorithm,algorithmicx,algpseudocode}
\usepackage{tensor}

\RequirePackage[OT1]{fontenc}
\RequirePackage{mathrsfs,amsfonts,amsmath,amssymb,subcaption,graphicx,epstopdf,enumerate,lmodern}
\usepackage{bm,color,fancyvrb,xcolor,mathtools}
\usepackage{bbm}
\def\eps{\varepsilon}
\def\lessim{\ \lower4pt\hbox{$
		\buildrel{\displaystyle <}\over\sim$}\ }
\def\gessim{\ \lower4pt\hbox{$\buildrel{\displaystyle >}
		\over\sim$}\ }

\usepackage{algorithm,algorithmicx,algpseudocode}
\usepackage{tensor}

\newtheorem{theorem}{Theorem}[section]
\newtheorem{proposition}[theorem]{Proposition}
\newtheorem{lemma}{Lemma}

\newtheorem{corollary}[theorem]{Corollary}

\DeclarePairedDelimiter{\norm}{\lVert}{\rVert}
\DeclarePairedDelimiter{\abs}{\lvert}{\rvert}

\providecommand{\abs}[1]{\left\lvert#1\right\rvert}
\providecommand{\norm}[1]{\left\lVert#1\right\rVert}
\renewcommand{\hat}{\widehat}

\renewcommand{\hat}{\widehat}

\providecommand{\abs}[1]{\left\lvert#1\right\rvert}
\providecommand{\norm}[1]{\left\lVert#1\right\rVert}
\renewcommand{\hat}{\widehat}

\renewcommand{\hat}{\widehat}

\newcommand{\bfm}[1]{\ensuremath{\mathbf{#1}}}

\def\ba{\bfm a}   \def\bA{\bfm A}  
\def\bb{\bfm b}     

\def\bc{\bfm c}     
   \def\bD{\bfm D}  
\def\be{\bfm e}     \def\EE{\mathbb{E}}

   \def\bI{\bfm I}

   \def\bM{\bfm M}

   \def\bP{\bfm P}  
     
   \def\bR{\bfm R}  \def\RR{\mathbb{R}}

\def\bu{\bfm u}   \def\bU{\bfm U}  
\def\bv{\bfm v}   \def\bV{\bfm V}  
\def\bw{\bfm w}     
\def\bx{\bfm x}

\def\calK{{\cal  K}} 
\def\calL{{\cal  L}} 
\def\calM{{\cal  M}}

\def\calP{{\cal  P}} 
\def\calQ{{\cal  Q}} 
 
\def\calS{{\cal  S}} 
 
\def\calU{{\cal  U}}

\def\scrT{\mathscr{T}}
\def\scrA{\mathscr{A}}

\def\scrE{\mathscr{E}}
\def\scrX{\mathscr{X}}
\def\scrU{\mathscr{U}}

\def\bfe{\mathbf{e}}

\DeclareMathOperator{\argmin}{argmin}

\def\eps{\varepsilon}

\def\newpage{\vfill\eject}

\usepackage{tikz}
\usetikzlibrary{decorations.pathreplacing,calc}

\makeatletter
\newcommand\footnoteref[1]{\protected@xdef\@thefnmark{\ref{#1}}\@footnotemark}
\newcommand\numberthis{\addtocounter{equation}{1}\tag{\theequation}}
\makeatother

\begin{document}
\title{Perturbation Bounds for (Nearly) Orthogonally Decomposable Tensors$^\ast$}
\author{Arnab Auddy and Ming Yuan$^{\dag}$\\
	Columbia University}
\date{(\today)}

\maketitle

\footnotetext[1]{
	This research was supported by NSF Grant DMS-2015285. Part of the work was done while the second author was visiting the Institute for Theoretical Studies at ETH Z\"urich, Switzerland, and he wishes to thank the institute for their hospitality.}
\footnotetext[2]{
	Address for Correspondence: Department of Statistics, Columbia University, 1255 Amsterdam Avenue, New York, NY 10027.}
\begin{abstract}
	We develop deterministic perturbation bounds for singular values and vectors of orthogonally decomposable tensors, in a spirit similar to classical results for matrices such as those due to Weyl, Davis, Kahan and Wedin. Our bounds demonstrate intriguing differences between matrices and higher-order tensors. Most notably, they indicate that for higher-order tensors perturbation affects each essential singular value/vector in isolation, and its effect on an essential singular vector does not depend on the multiplicity of its corresponding singular value or its distance from other singular values. Our results can be readily applied and provide a unified treatment to many different problems in statistics and machine learning involving spectral learning of higher-order orthogonally decomposable tensors. In particular, we illustrate the implications of our bounds in the context of high dimensional tensor SVD problem, and how it can be used to derive optimal rates of convergence for spectral learning.
\end{abstract}
\newpage

\section{Introduction}
Singular value decomposition (SVD) is routinely performed to process data organized in the form of matrices, thanks to its optimality for low-rank approximation, and relationship with principal component analysis; and perturbation analysis of SVD plays a central role in studying the performance of these procedures. More and more often, however, multidimensional data in the form of higher-order tensors arise in applications. While higher-order tensors provide us a more versatile tool to encode complex relationships among variables, how to perform decompositions similar to SVD and how these decompositions behave under perturbation are often the most fundamental issues in these applications. In general, decomposition of higher-order tensors is rather delicate and poses both conceptual and computational challenges. See \cite{kolda2009tensor, cichocki2015tensor} for recent surveys of some of the difficulties as well as existing techniques and algorithms to tackle them. In particular, we shall focus here on a class of tensors that allows for direct generalization of SVD. The so-called orthogonally decomposable (odeco) tensors have been previously studied by \cite{kolda2001ortho, chen2009tensor, robeva2016orthogonal, belkin2018eigenvectors} among others, and commonly used in high dimensional data analysis \citep{anand2014tensor, anand2014sample, anandkumar2014guaranteed, liu2017characterizing}. The main goal of this work is to study the effect of perturbation on the singular values and vectors of an odeco tensor or odeco approximations of a nearly odeco tensor, and demonstrate how it could provide a powerful and unifying treatment to many different problems in high dimensional data analysis.

More specifically, an orthogonally decomposable tensor $\scrT\in {\mathbb R}^{d\times\cdots\times d}$ can be written as
\begin{equation}
	\label{eq:odec1}
	\scrT=\sum_{i=1}^d\lambda_k\bu_k^{(1)}\otimes\dots\otimes\bu_k^{(p)}
\end{equation}
where $\lambda_1\ge \lambda_2\ge\cdots\lambda_d\ge 0$, and the matrices $\bU^{(q)}=[\bu_1^{(q)}\,\dots\,\bu_d^{(q)}]\in \RR^{d\times d}$ for $1\le q\le p$ are orthonormal. It is well known that such a decomposition is essentially unique. Here we are interested in its stability: how perturbation to $\scrT$ may affect our ability to reconstruct the spectral parameters $\lambda_k$s and $\bu^{(q)}_k$s, which we shall refer to as the essential singular values and vectors, or simply singular values and vectors when no confusion occurs, of $\scrT$. See Section 2 for discussion of singular values and vectors for tensors. 

Perturbation theory of this nature is well-developed in the case of matrices ($p=2$) and can be traced back to the classical works  \cite{weyl1912asymptotische}, \cite{davis1970rotation} and  \cite{wedin1972perturbation}. See, e.g., \cite{stewart1990matrix} for a comprehensive survey. These results provide the essential tools for numerous applications in various scientific and engineering domains. As multilinear arrays appear more and more often in these applications, many attempts have been made to develop similar tools for higher order tensors in recent years. Because of the unique challenges associated with higher order tensors, most if not all existing studies along this direction customize their analysis and hence the resulting bounds for a specific algorithm or method. See, e.g., \cite{anand2014tensor, mu2015successive, mu2017greedy, belkin2018eigenvectors}. The aim of this article is to fill in the important step of providing universal perturbation bounds that is in the same spirit as matrix perturbation analysis and independent of a specific algorithm. Doing so not only provides universal perturbation bounds that can be useful for all these applications together, but also allows us to recognize the fundamental similarities and differences between matrices and higher order tensor from yet another perspective.

In particular, consider, in addition to $\scrT$, a second odeco tensor $\tilde{\scrT}$:
\begin{equation}
	\label{eq:odec2}
	\tilde{\scrT}=\sum_{i=1}^d\tilde{\lambda}_k\tilde{\bu}_k^{(1)}\otimes\dots\otimes\tilde{\bu}_k^{(p)}.
\end{equation}
We are interested in how the differences between the two sets of values $\lambda_k$s and $\tilde{\lambda}_k$s as well as the vectors $\bu^{(q)}_k$s and $\tilde{\bu}^{(q)}_k$s are characterized by the spectral norm of the difference $\scrT-\tilde{\scrT}$, in a spirit similar to classical results for matrices. We show that there exist a \emph{numerical} constant $C_0\ge 1$ and a permutation $\pi: [d]\to[d]$ such that for all $k=1,\ldots,d$,
\begin{equation}
	\label{eq:weyl}
	|\lambda_k-\tilde{\lambda}_{\pi(k)}|\le C_0\|\scrT-\tilde{\scrT}\|,
\end{equation}
and 
\begin{equation}
	\label{eq:wedin}
	\max_{1\le q\le p}\sin\angle (\bu^{(q)}_k,\tilde{\bu}^{(q)}_{\pi(k)})\le C_0\cdot{\|\scrT-\tilde{\scrT}\|\over \lambda_k},
\end{equation}
under the convention that $1/0=+\infty$. Here and in what follows $\angle (\bu,\tilde{\bu})$ is the angle between two vectors $\bu$ and $\tilde{\bu}$ taking value in $[0,\pi/2]$, and the spectral norm of a tensor $\scrA\in \RR^{d\times\cdots\times d}$ is defined by
$$
\|\scrA\|=\max_{\bu^{(q)}\in \calS^{d-1}} \langle \scrA, \bu^{(1)}\otimes\cdots\otimes \bu^{(p)}\rangle.
$$
and $\calS^{d-1}$ denotes the unit sphere in $\RR^{d}$.

We want to emphasize that the constant $C_0$ in \eqref{eq:weyl} and \eqref{eq:wedin} is absolute and independent of $\scrT$, $\tilde{\scrT}$, and their dimensionality $d$ or $p$. This is especially relevant and important when dealing with high dimensional problems either statistically or numerically, as we shall demonstrate in Section \ref{sec:app}. In particular, we can take the constant $C_0$ above to be 17. We did not attempt to optimize this constant to its fullest extent, as a much better value can be provided if there is more information on how $\scrT$ and $\tilde{\scrT}$ are related: if a singular value $\lambda_k$ is sufficiently large relative to the size of perturbation $\|\scrT-\tilde{\scrT}\|$, then we can take the constant $C_0=1$ in \eqref{eq:weyl} and arbitrarily close to $1$ in \eqref{eq:wedin}.  In particular, under infinitesimal perturbations such that $\|\scrT-\tilde{\scrT}\|=o(\lambda_k)$, we have
\begin{equation}
	\label{eq:weylpert}
	|\lambda_k-\tilde{\lambda}_{\pi(k)}|\le \|\tilde{\scrT}-\scrT\|,
\end{equation}
and
\begin{equation}
	\label{eq:wedinpert}
	\max_{1\le q\le p}\sin\angle (\bu^{(q)}_k,\tilde{\bu}^{(q)}_{\pi(k)})\le {\|\tilde{\scrT}-\scrT\|\over \lambda_k}+o\left({\|\tilde{\scrT}-\scrT\|\over \lambda_k}\right).
\end{equation}
Both bounds are sharp in that the leading terms cannot be further improved. This is clear by considering two rank-one tensors differing only in the nonzero singular value or in one of its corresponding singular vectors. 

Note that every matrix is odeco, \eqref{eq:weyl}-\eqref{eq:wedinpert} therefore directly extends classical results for matrices ($p=2$) by  \cite{weyl1912asymptotische}, \cite{davis1970rotation},  \cite{wedin1972perturbation} among others. However, in spite of the similarity in appearance, there are also crucial distinctions between matrices and higher order tensors ($p\ge 3$). In particular, the $\sin\Theta$ theorems of Davis-Kahan-Wedin bound the perturbation effect on the $k$th singular vector by $C\|\tilde{\scrT}-\scrT\|/\min_{j\neq k}|\lambda_j-\lambda_k|$. The dependence on the gap $\min_{j\neq k}|\lambda_j-\lambda_k|$ between $\lambda_k$ and other singular values is unavoidable for matrices. This is not the case for higher-order odeco tensors where perturbation affects the singular vectors in separation. Indeed the crux of our technical argument is devoted to proving this by careful control of spillover effect of not knowing other singular tuples on $(\lambda_k,\bu_k^{(1)},\ldots,\bu_k^{(p)})$ and showing that the approximation errors do not accumulate.

In general, a perturbed odeco tensor $\scrX=\scrT+\scrE$ may no longer be odeco and hence it may not be possible to match its singular value/vector tuples with the essential singular value/vector tuples of the unperturbed odeco tensor. To overcome this obstacle, we shall consider instead an odeco approximation $\tilde{\scrT}$ to $\scrX$ such that $\|\tilde{\scrT}-\scrX\|\le C_1\|\scrE\|$ for some constant $C_1>0$. By triangular inequality,
\begin{equation}
	\label{eq:aprxodec}
	\|\scrT-\tilde{\scrT}\|\le\|\scrT-\scrX\|+ \|\tilde{\scrT}-\scrX\|=(C_1+1)\|\scrE\|.
\end{equation}
Then \eqref{eq:weyl} and \eqref{eq:wedin} imply that
$$
|\lambda_k-\tilde{\lambda}_{\pi(k)}|\le C_0(C_1+1)\|\scrE\|
$$
and
$$
\max_{1\le q\le p}\sin\angle (\bu^{(q)}_k,\tilde{\bu}^{(q)}_{\pi(k)})\le C_0(C_1+1)\cdot{\|\scrE\|\over \lambda_k},
$$
where $(\tilde{\lambda}_k,\bu^{(q)}_k: 1\le q\le p)$s are the (essential) singular value/vectors tuple of $\tilde{\scrT}$. These bounds complement the well known identifiability of odeco decomposition that states if $\scrE=0$, then all $\bu^{(q)}_k$s are uniquely defined. When $\scrE\neq 0$, $\scrX$ is not necessarily odeco but our results indicate that when the perturbation is small, any ``reasonable'' odeco approximation of $\scrX$ would have ``similar'' essential singular values and vectors. This is more general than identifiability and in fact characterizes the \emph{stability} of odeco decomposition or the local geometry of the space of odeco tensors.

It is natural to consider deriving perturbation bounds for higher order tensors by first flattening them into matrices and then applying the existing bounds for matrices. As we shall show, such a na\"ive approach is suboptimal in that it inevitably leads to perturbation bounds in terms of the matricized spectral norm. Although it is possible to further bound matricized spectral norms using tensor spectral norms, it leads to an extra multiplicative factor depending on the dimension ($d$) polynomially, and makes the resulting bounds unsuitable for applications in high dimensional problems. Our results demonstrate that there could be tremendous gain by treating higher order tensors as tensors instead of matrices. Moreover, the matricization approach fails to yield meaningful perturbation bounds for an essential singular vector when the corresponding singnular value is not simple, e.g., when $\lambda_k=\lambda_{k+1}$. We summarize classical perturbation bounds for matrices and those we establish for odeco tensors in Table \ref{tab:foo}.

\begin{table}[htbp]
	\begin{center}
		\caption{Comparison of Perturbation Bounds in terms of $\|\scrE\|$, up to a constant factors.}  \label{tab:foo}
	\end{center}
	\begin{center}
		\begin{tabular}{c|cc} \hline\hline
			& Singular values & Singular vectors\\
			& ($|\lambda_k-\tilde{\lambda}_{\pi(k)}|$) & ($\sin\angle(\bu^{(q)}_k,\tilde{\bu}^{(q)}_{\pi(k)})$\\
			\hline
			& & \\
			Matrix & $\|\scrE\|$ & $\|\scrE\|\over \min\{\lambda_{k-1}-\lambda_k, \lambda_k-\lambda_{k+1}\}$\\
			\hline
			& \multicolumn{2}{c}{with matricization}\\
			Odeco Tensor & ${\rm poly}(d)\cdot\|\scrE\|$  & ${{\rm poly}(d)\cdot\|\scrE\|\over \min\{\lambda_{k-1}-\lambda_k, \lambda_k-\lambda_{k+1}\}}$\\
			\cline{2-3}
			& \multicolumn{2}{c}{without matricization}\\
			& $\|\scrE\|$  & $\|\scrE\|\over \lambda_k$\\
			\hline
		\end{tabular}
	\end{center}
\end{table}

Given the importance of perturbation analysis in fields such as machine learning, numerical analysis, and statistics, it is conceivable that our analysis and algorithms can prove useful in many situations. For illustration, we shall consider a specific example, namely high dimensional tensor SVD. Our general perturbation bound immediately leads to new insights to the problem. In particular, we establish minimax optimal rates for estimating the singular vectors of an odeco tensor when contaminated with Gaussian noise. Our result indicates that any of its singular vectors can be estimated as well as if all other singular values are zero, or in other words, as in the rank one case.

Our development is related to the fast-growing literature on using tensor methods in statistics and machine learning. In particular, there is a fruitful line of research in developing algorithm dependent bounds for odeco tensors. In these applications, we always encounter a noisy version of the signal tensor and dimension-independent perturbation bounds of the singular values and vectors are the most critical tool in the analysis. See \cite{janzamin2019spectral} for a recent survey. A significant conceptual difference between these bounds and those classical perturbation bounds for matrices is that they are specific to the algorithms used in computing $\tilde{\scrT}$ or equivalently its SVD. The perturbation bounds we provide complement these earlier developments in a number of ways. First of all, our bounds could be readily used for perturbation analysis of any algorithm that produces an odeco approximation, allowing us to derive bounds on the singular values and vectors from those on the approximation error of the tensor itself. As such we do not rely on the specific form of the error tensor (as in \cite{anand2014tensor}, \cite{anand2014sample} or \cite{belkin2018eigenvectors}) and also have the weakest possible assumption on the signal to noise ratio. On the other hand, our bounds can also serve as a benchmark on how well any procedure, computationally feasible or not, could perform. Indeed as we can see from the high dimensional data analysis example, our perturbation bounds often yield tight information theoretical limits for statistical inferences. In fact a similar rate optimality continues to hold for a number of other tensor data problems.

The rest of the paper is organized as follows. In the next section, we derive perturbation bounds for a pair of odeco tensors. Section 3 extends these bounds to nearly odeco tensors. 
Proofs of the main results are presented in Section 4.

\section{Perturbation Bounds between Odeco Tensors}
In this section, we shall consider perturbation analysis for a pair of odeco tensors. We first review some basic properties of odeco tensors and then consider two ways to derive perturbation bounds between a pair of odeco tensor: one through matricization and the other by treating tensors as tensors. While we focus primarily on the so-called essential singular values and vectors, we shall also brief discuss how our techniques may be used to derive perturbation bounds for general singular values and vectors.

\subsection{Odeco Tensors}

We say a $p$th order tensor $\scrT\in \RR^{d_1\times\cdots\times d_p}$ is odeco if it can be expressed as
\begin{equation}
	\label{eq:svd}
	\scrT=\sum_{k=1}^{d_{\min}}\lambda_k \bu_k^{(1)}\otimes\cdots\otimes\bu_k^{(p)}
\end{equation}
for some scalars $\lambda_1\ge \ldots \ge \lambda_{d_{\min}}\ge 0$ and unit vectors $\bu_k^{(q)}$s such that $\langle \bu_{k_1}^{(q)}, \bu_{k_2}^{(q)}\rangle =\delta_{k_1k_2}$ where $d_{\min}=\min\{d_1,\ldots,d_p\}$ and $\delta$ is the Kronecker's delta. Note that there is no loss of generality in assuming that $\lambda_k$s are nonnegative as we can flip the sign of $\bu_k^{(q)}$s accordingly.  See, e.g., \cite{kolda2001ortho, robeva2016orthogonal, robeva2017singular} for further discussion of orthogonally decomposable tensors. For brevity, we shall write
$$
\scrT=[\{\lambda_k: 1\le k\le d_{\min}\}; \bU^{(1)},\ldots,\bU^{(p)}]
$$
if \eqref{eq:svd} holds. Here $\bU^{(q)}\in \RR^{d_q\times d_{\min}}$ with $\bu_k^{(q)}$ as its $k$th column. 

Recall that, in general, singular values and vectors for a tensor $\scrT$ are defined as tuples $(\lambda,\bv^{(1)},\dots,\bv^{(p)})\in \RR\times \RR^{d_1}\times\dots\times \RR^{d_p}$ such that $\|\bv^{(q)}\|=1$ and  
$$
\scrT\times_{j\neq q}\bv^{(j)}=\lambda \bv^{(q)}\quad{\rm for}\quad \,q=1,\dots,p.
$$
See, e.g., \cite{hackbusch2012tensor, qi2017tensor} for further details. For odeco tensors, all possible singular values and vectors of $\scrT$ can be characterized by $\lambda_k$s and $\bu^{(q)}_k$s: if $\lambda_r>0=\lambda_{r+1}$, then the real singular values and singular vectors of $\scrT$ are either tuples $(\lambda,\bv^{(1)},\dots,\bv^{(p)})$ of the form:
$$
\lambda = \left(\sum_{k\in S}\dfrac{1}{\lambda_k^{\tfrac{2}{p-2}}}\right)^{-\tfrac{p-2}{2}},\qquad
\,\langle \bv^{(q)},\,\bu^{(q)}_k\rangle
=\begin{cases}
	\chi_k^{(q)}\left(\lambda\over\lambda_k\right)^{1/(p-2)}\,&\text{if }k\in S\\
	0\quad&\text{otherwise}
\end{cases}$$
where $S\subset[r]$, $S\neq \emptyset$, and $\chi_k^{(q)}\in\{+1,-1\}$ satisfy $\displaystyle\prod_{q=2}^p\chi_k^{(q)}=1$ for all $1\le k\le r $; or $\lambda=0$, and $(\bv^{(1)},\dots,\bv^{(p)})$ are such that for every $1\le k\le d_{\min}$, there exist at least two values of $q\in\{1,\dots,p\}$ with $\langle \bv^{(q)},\bu^{(q)}_k\rangle=0$. See \cite{robeva2017singular} for details. 

In this article, we shall focus primarily on the perturbation of singular value/vector tuples $(\lambda_k,\bu_k^{(q)}: 1\le q\le p)$s and refer to them as the \emph{essential singular values and vectors}, or with some abuse of notation singular values and vectors for short, of $\scrT$ with the exception of Section 2.5 where we shall explicitly discuss how perturbation bounds for other singular value and vectors can be obtained.

In the case when all $\lambda_k$s are distinct, the essential singular values and vectors can be identified by the so-called higher-order SVD (HOSVD) which applies SVD after flattening a higher-order tensor to a matrix, for example, by collapsing all indices except the first one. See, e.g., \cite{de2000multilinear, de2000best}. However, this is not the case when the singular values have multiplicity more than one (i.e., when $\lambda_k=\lambda_{k+1}$ for some $k$ on the right hand side of \eqref{eq:svd}) since HOSVD can only identify the singular space associated with a singular value. This subtle difference also has important practical implications. In general, the essential singular vectors of odeco tensors cannot be computed via HOSVD unless all singular values are distinct.

Nonetheless computing the essential singular value/vectors for an odeco tensor is tractable. For example, it can be computed via Jennrich's algorithm when $p=3$. See, e.g., \cite{harshman1970foundations, leurgans1993decomposition}. More generally, efficient algorithms also exist to take full advantage of the orthogonal structure. In particular, if an odeco tensor is symmetric so that $d_1=\cdots=d_p=:d$, and $\bu_k^{(1)}=\cdots=\bu_k^{(p)}=:\bu_k$ for all $k=1,\ldots,d$, \cite{belkin2018eigenvectors} showed that $\pm\bu_k$s are the only local maxima of
$$F(\ba):=|\langle \scrT, \ba\otimes\cdots\otimes \ba\rangle|$$
over $\calS^{d-1}$. In addition, there is a full measure set $\calU\subset \calS^{d-1}$ such that a gradient iteration algorithm with initial value arbitrarily chosen from $\calU$ converges to one of the $\bu_k$s. In light of these properties, one can enumerate all the essential singular values and essential singular vectors by repeatedly applying the gradient iteration algorithm with an initial value randomly chosen from the orthogonal complement of the linear space spanned by those already identified local maxima. For this property, these vectors are also called robust singular vectors in the literature \citep[see][]{anand2014tensor}. Interested readers are referred to \cite{belkin2018eigenvectors} for further details. From a slightly different perspective, \cite{hashemi2018spectral} study trivariate analytic functions that are two way odeco, and describe how CP decomposition enables one to derive low rank approximation of such functions. 

The argument presented in \cite{belkin2018eigenvectors} relies heavily on the hidden convexity of $F$, which no longer holds when $\scrT$ is not symmetric. However, their main observations remain valid for general odeco tensors. More specifically, write
\begin{equation}
	\label{eq:defF}
	F(\ba^{(1)},\ldots, \ba^{(p)}):=|\langle \scrT, \ba^{(1)}\otimes\cdots\otimes \ba^{(p)}\rangle|
\end{equation}
with slight abuse of notation. Denote by
\begin{equation}
	\label{eq:defG}
	G(\ba^{(1)},\ldots, \ba^{(p)}):=\left({\scrT\times_2\ba^{(2)}\cdots\times_p\ba^{(p)}\over\|\scrT\times_2\ba^{(2)}\cdots\times_p\ba^{(p)}\|},\ldots,{\scrT\times_1\ba^{(1)}\cdots\times_{p-1}\ba^{(p-1)}\over\|\scrT\times_1\ba^{(1)}\cdots\times_{p-1}\ba^{(p-1)}\|}\right).
\end{equation}
the gradient iteration function for $F$ so that
$$G_n=\underbrace{G\circ G\circ\cdots\circ G}_{n{\rm\ times}}$$
maps from a set of initial values to the output from running the gradient iteration $n$ times. Similar to the symmetric case, we have the following result for general odeco tensors:

\begin{theorem}\label{th:comp}
	Let $\scrT$ be an odeco tensor, and $F$ and $G$ be defined by \eqref{eq:defF} and \eqref{eq:defG} respectively. Then the set $\{(\pm\bu_k^{(1)},\ldots,\pm\bu_k^{(p)}): \lambda_k>0\}$ is a complete enumeration of all local maxima of $F$. Moreover, there exists a full measure set $\calU\subset \calS^{d_1-1}\times\cdots\times\calS^{d_p-1}$ such that for any $(\ba^{(1)},\ldots,\ba^{(p)})\in \calU$, $G_n(\ba^{(1)},\ldots,\ba^{(p)})\to (\sigma_1\bu_k^{(1)},\ldots,\sigma_p\bu_k^{(p)})$ as $n\to \infty$,  for some $1\le k\le d_{\min}$, and $\sigma_1,\ldots,\sigma_p\in \{\pm 1\}$.
\end{theorem}

The main architect of the proof of Theorem \ref{th:comp} is similar to that for symmetric cases. See, e.g., \cite{belkin2018eigenvectors}. For completeness, a detailed proof is included in the Appendix. In light of  Theorem \ref{th:comp}, we can then compute all the essential singular value/vector tuples of an odeco tensor sequentially by applying gradient iterations and random initializations, in the same manner as the symmetric case.

\subsection{Perturbation Bounds via Matricization}
Let $\scrT$ and $\tilde{\scrT}$ be two odeco tensors:
$$
\scrT=[\{\lambda_k: 1\le k\le d_{\min}\}; \bU^{(1)},\ldots,\bU^{(p)}],
$$
and
$$
\tilde{\scrT}=[\{\tilde{\lambda}_k: 1\le k\le d_{\min}\}; \tilde{\bU}^{(1)},\ldots,\tilde{\bU}^{(p)}],
$$
We are interested in characterizing the difference between the two sets of singular values and vectors in terms of the ``perturbation'' $\tilde{\scrT}-\scrT$.

It is instructive to first briefly review classical results in the matrix case, i.e., $p=2$. Note that every matrix is odeco. Perturbation analysis of the singular vectors and spaces for matrices is well studied. See, e.g., \cite{bhatia1987perturbation, stewart1990matrix}, and references therein. In particular, Weyl's perturbation theorem indicates that
\begin{equation}
	\label{eq:matweyl0}
	\max_{1\le k\le d_{\min}} |\lambda_k-\tilde{\lambda}_k|\le \|\scrT-\tilde{\scrT}\|.
\end{equation}
When a singular value $\lambda_k$ has multiplicity more than one, its singular space has dimension more than one and singular vectors $\bu_k$ and $\bv_k$ are no longer uniquely identifiable. But if it is simple, i.e., $\lambda_{k-1}>\lambda_k>\lambda_{k+1}$, then the Davis-Kahan-Wedin $\sin\Theta$ theorem states that
\begin{equation}
	\label{eq:davis0}
	\sin\angle(\bu_k^{(q)},\tilde{\bu}_k^{(q)})\le {\|\scrT-\tilde{\scrT}\|\over \min\{\tilde{\lambda}_{k-1}-\lambda_k,\lambda_k-\tilde{\lambda}_{k+1}\}},
\end{equation}
provided that the denominator on the righthand side is positive. It is oftentimes more convenient to consider a modified version of the above bound for the singular vectors in terms of the gap between singular values of $\scrT$:
\begin{equation}
	\label{eq:davis}
	\sin\angle(\bu_k^{(q)},\tilde{\bu}_k^{(q)})\le {2\|\scrT-\tilde{\scrT}\|\over \min\{\lambda_{k-1}-\lambda_k,\lambda_k-\lambda_{k+1}\}},
\end{equation}
which follows immediately from \eqref{eq:matweyl0} and \eqref{eq:davis0}. To see this, note that \eqref{eq:davis} holds trivially if $\|\scrT-\tilde{\scrT}\|\ge \min\{\lambda_{k-1}-\lambda_k,\lambda_k-\lambda_{k+1}\}/2$. On the other hand, if $\|\scrT-\tilde{\scrT}\|< \min\{\lambda_{k-1}-\lambda_k,\lambda_k-\lambda_{k+1}\}/2$, it follows from \eqref{eq:matweyl0} that
\begin{eqnarray*}
	\min\{\tilde{\lambda}_{k-1}-\lambda_k,\lambda_k-\tilde{\lambda}_{k+1}\}&\ge& \min\{\lambda_{k-1}-\lambda_k,\lambda_k-\lambda_{k+1}\}-\|\scrT-\tilde{\scrT}\|\\
	&\ge& {1\over 2}\min\{\lambda_{k-1}-\lambda_k,\lambda_k-\lambda_{k+1}\},
\end{eqnarray*}
and therefore \eqref{eq:davis} follows from \eqref{eq:davis0}.

It is worth noting that the dependence of any general perturbation bounds for singular vectors on the gap between singular values is unavoidable for matrices and can be illustrated by the following simple example from \cite{bhatia2013matrix}:
\begin{equation}
	\label{eq:bhatia}
	\scrT=\left(\begin{array}{cc}1+\delta& 0\\ 0& 1-\delta\end{array}\right),\qquad {\rm and}\qquad \tilde{\scrT}=\left(\begin{array}{cc}1 & \delta\\ \delta& 1\end{array}\right).
\end{equation}
It is not hard to see that $\|\scrT-\tilde{\scrT}\|=\sqrt{2}\delta$ and can be made arbitrarily small at the choice of $\delta>0$. Yet the singular vectors of $\scrT$ and $\tilde{\scrT}$ are $\{(0,1)^\top,(1,0)^\top\}$ and $\{(1/\sqrt{2}, 1/\sqrt{2})^\top,(1/\sqrt{2}, -1/\sqrt{2})^\top\}$ respectively so that
$$
\sin\angle(\bu_k^{(q)},\tilde{\bu}_k^{(q)})={\|\scrT-\tilde{\scrT}\|\over \lambda_1-\lambda_2},
$$
for $k=1,2$ and $q=1,2$.

These classical perturbation bounds can be applied to higher-order tensors using matricization or flattening, as for HOSVD. More precisely, write ${\sf Mat}_q: \RR^{d_1\times\cdots\times d_p}\to$
\\$ \RR^{d_q\times d_{-q}}$ by collapsing all indices other than the $q$th one and therefore converting a $p$th order tensor into a $d_q\times d_{-q}$ matrix where $d_{-q}=d_1\cdots d_{q-1}d_{q+1}\cdots d_p$. For an odeco tensor $\scrT$, its SVD determines that of ${\sf Mat}_q(\scrT)$. More specifically,
$$
{\sf Mat}_q(\scrT)=\bU^{(q)}({\rm diag}(\lambda_1,\ldots,\lambda_{d_{\min}}))(\bV^{(q)})^\top,
$$
where
$$
\bV^{(q)}=\bU^{(1)}\odot\cdots\odot\bU^{(q-1)}\odot\bU^{(q+1)}\odot\cdots\odot\bU^{(p)}.
$$
Here $\odot$ stands for the Khatri-Rao product. This, in light of \eqref{eq:matweyl0} and \eqref{eq:davis}, immediately implies that
\begin{proposition}\label{pr:matricize}
	Let $\scrT$ and $\tilde{\scrT}$ be two $d_1\times\cdots\times d_p$ odeco tensors with SVD:
	$$
	\scrT=[\{\lambda_k: 1\le k\le d_{\min}\}; \bU^{(1)},\ldots,\bU^{(p)}],
	$$
	and
	$$
	\tilde{\scrT}=[\{\tilde{\lambda}_k: 1\le k\le d_{\min}\}; \tilde{\bU}^{(1)},\ldots,\tilde{\bU}^{(p)}],
	$$
	respectively where $d_{\min}=\min\{d_1,\ldots, d_p\}$. If $\lambda_k$ is simple, then
	$$
	|\lambda_k-\tilde{\lambda}_k|\le \min_{1\le q\le p}\|{\sf Mat}_q(\scrT)-{\sf Mat}_q(\tilde{\scrT})\|,
	$$
	and
	$$
	\sin\angle(\bu_k^{(q)},\tilde{\bu}_k^{(q)})\le {2\|{\sf Mat}_q(\scrT)-{\sf Mat}_q(\tilde{\scrT})\|\over \min\{\lambda_{k-1}-\lambda_k,\lambda_k-\lambda_{k+1}\}}.
	$$
\end{proposition}
These bounds, however, are suboptimal and can be significantly improved in a couple of directions that highlight fundamental differences between matrices and higher-order tensors.

First of all, we can derive perturbation bounds in terms of the tensor operator norm $\|\scrT-\tilde{\scrT}\|$. Although it is true that $\|\scrA\|=\|{\sf Mat}_q(\scrA)\|$ for $q=1,\ldots, p$ for an odeco tensor $\scrA$, the difference between two odeco tensors is not necessarily odeco and as a result $\|\scrT-\tilde{\scrT}\|$ and $\|{\sf Mat}_q(\scrT)-{\sf Mat}_q(\tilde{\scrT})\|$ can be quite different. As a simple example, consider the case when $\scrT=\bu\otimes\bu\otimes \bu$ and $\tilde{\scrT}=\bu\otimes\bv\otimes\bv$ where $\bu=(0,1)^\top$ and $\bv=(1,0)^\top$. It is easy to see that $\|\scrT-\tilde{\scrT}\|=1$ and yet $\|{\sf Mat}_1(\scrT)-{\sf Mat}_1(\tilde{\scrT})\|=\sqrt{2}$. Note that we can always bound
$$
\|{\sf Mat}_q(\scrA)\|\le C_d\|\scrA\|
$$
for a multiplicative factor $C_d$ that depends on the dimension $d$ so that we can translate the aforementioned bounds on $\tilde{\lambda}_k$ and $\tilde{\bu}_k$ in terms of tensor spectral norm $\|\scrT-\tilde{\scrT}\|$. This, however, is rather unsatisfactory when it comes to high dimensional problems ($d$ is large) as $C_d\ge \sqrt{d-1}$ as the following example shows.

Let
\begin{equation}\label{eq:ortho_cntrex}
	\scrT=\lambda \displaystyle\sum_{i=1}^{d-1}\be_i\otimes \be_i\otimes \be_i
	\quad
	{\rm and }
	\quad
	\tilde{\scrT}=\lambda \displaystyle\sum_{i=1}^{d-1}({\be}_i+\bv)\otimes \be_i\otimes \be_i
\end{equation}
where
$$\bv=\dfrac{1}{\sqrt{d-1}}\be_d-\dfrac{1}{d-1}\left(\be_1+\dots+\be_{d-1}\right).$$ 
It is easy to see that both are odeco. Note that $\scrT-\tilde{\scrT}=\lambda\displaystyle\sum_{i=1}^{d-1}\bv\otimes \be_i\otimes \be_i$ and hence
$$\|\scrT-\tilde{\scrT}\|=\lambda\underset{\ba,\bb,\bc\in\calS^{d-1}}{\sup}\langle \bv,\ba\rangle \sum_{i=1}^{d-1}b_ic_i=\lambda \|\bv\|.$$
On the other hand,
$${\sf Mat}_1(\scrT-\tilde{\scrT})=\lambda\displaystyle \sum_{i=1}^{d-1}\bv\left(\be_i\odot \be_i\right)^\top.$$
With the two unit vectors
$$\ba=\bv/\|\bv\|\qquad {\rm and} \qquad \bb=\dfrac{1}{\sqrt{d-1}}\displaystyle\sum_{i=1}^{d-1}\left(\be_i\odot \be_i\right),$$
we have 
$$\|{\sf Mat}_1(\scrT-\tilde{\scrT})\|\ge \ba^\top{\sf Mat}_1(\scrT-\tilde{\scrT})\bb =\lambda \|\bv\|\sqrt{d-1}= \sqrt{d-1}\|\scrT-\tilde{\scrT}\|.$$

This immediately suggests that the constant $C_d$ in the bound derived from matricization necessarily diverges as $d$ increases when $p>2$, and this renders the perturbation bounds derived from matricization ineffective in many applications where the focus is on pinpointing the effect of increasing dimensionality. Fortunately, as we shall show in the next subsection, much sharper perturbation bounds in terms of $\|\scrT-\tilde{\scrT}\|$ are available.

Perhaps more importantly, another undesirable aspect of the aforementioned perturbation bounds for higher-order odeco tensors is the dependence on the gap between singular values. For matrices it is only meaningful to talk about singular spaces when a singular value is not simple, and the aforementioned bounds for $\sin\angle(\bu_k^{(q)},\tilde{\bu}_k^{(q)})$ does not tell us anything about the perturbation of the singular vectors at all when a singular value is not simple even though all essential singular vectors are identifiable for higher-order odeco tensors regardless of the multiplicity of its singular values. Indeed, as we shall show, that the gap $\min\{\lambda_{k-1}-\lambda_k,\lambda_k-\lambda_{k+1}\}$ is irrelevant for perturbation analysis of a higher order odeco tensors, and perturbation of each singular vectors is independent of other singular values.

\subsection{Perturbation Bounds for Odeco Tensors}

To appreciate the difference in perturbation effect between matrices and higher-order tensors, we first take a look at the Weyl's bound for singular values which states that, in the matrix case, i.e., $p=2$,
\begin{equation}
	\label{eq:matweyl}
	\max_{1\le k\le d}|\lambda_k-\tilde{\lambda}_k|\le \|\scrT-\tilde{\scrT}\|.
\end{equation}
More generally, when $p$ is even, asymptotic bounds for simple singular values under infinitesimal perturbation have been studied recently by \cite{che2016perturbation}. Their result implies that, in our notation, if $p$ is even and a simple singular value $\lambda_j$ is sufficiently far away from $\lambda_{j-1}$ and $\lambda_{j+1}$, then
$$
|\tilde{\lambda}_j-\lambda_j|\le \|\tilde{\scrT}-\scrT\|+O(\|\tilde{\scrT}-\scrT\|^2),
$$
as $\|\tilde{\scrT}-\scrT\|\to 0$. This appears to suggest that it is plausible that \eqref{eq:matweyl} could continue to hold for higher-order odeco tensors. Unfortunately, this is not the case and \eqref{eq:matweyl} does not hold in general for higher-order odeco tensors. To see this, let
$$
\scrT=2\bfe_1\otimes \bfe_1\otimes \bfe_1
$$
and
$$
\tilde{\scrT}=(\bfe_1+\bfe_2)\otimes (\bfe_1+\bfe_2)\otimes (\bfe_1+\bfe_2)+(\bfe_1-\bfe_2)\otimes (\bfe_1-\bfe_2)\otimes (\bfe_1-\bfe_2).
$$
Obviously $(\lambda_1,\lambda_2)=(2,0)$ and $(\tilde{\lambda}_1,\tilde{\lambda}_2)=(2\sqrt{2},2\sqrt{2})$ so that
$$
\max\{|\lambda_1-\tilde{\lambda}_1|,|\lambda_2-\tilde{\lambda}_2|\}=2\sqrt{2}.
$$
On the other hand, as shown by \cite{yuan2016tensor}
$$
\|\tilde{\scrT}-\scrT\|=2\|\bfe_2\otimes \bfe_2\otimes \bfe_1+\bfe_2\otimes\bfe_1\otimes\bfe_2+\bfe_1\otimes\bfe_2\otimes\bfe_2\|=4/\sqrt{3}<2\sqrt{2},
$$
invalidating \eqref{eq:matweyl}.

At a more fundamental level, for matrices, Weyl's bound can be viewed as a consequence of Courant-Fischer-Weyl min-max principle which states that
\begin{equation}
	\label{eq:minmax1}
	\lambda_k=\min_{S: {\rm dim}(S)=d_1-k+1}\max_{\substack{\bx^{(1)}\in \calS^{d_1-1}\cap S\\ \bx^{(2)}\in \calS^{d_2-1}}}\langle \scrT, \bx^{(1)}\otimes\bx^{(2)}\rangle,
\end{equation}
and
\begin{equation}
	\label{eq:minmax2}
	\lambda_k=\max_{S: {\rm dim}(S)=k}\min_{\bx^{(1)}\in \calS^{d_1-1}\cap S}\max_{\bx^{(2)}\in \calS^{d_2-1}}\langle \scrT, \bx^{(1)}\otimes\bx^{(2)}\rangle.
\end{equation}
Similar characterizations, however, do not hold for higher-order tensors. As an example, consider a $p$th order odeco tensor of dimension $d\times\cdots\times d$ and with equal singular values. The following proposition shows that neither \eqref{eq:minmax1} nor \eqref{eq:minmax2} holds, in particular for the smallest essential singular value $\lambda_d$ where the righthand side of both equations can be expressed as
$$
\min_{\bx^{(1)}\in \calS^{d-1}}\max_{\bx^{(2)}, \ldots, \bx^{(p)}\in \calS^{d-1}}\langle \scrT, \bx^{(1)}\otimes\cdots\otimes\bx^{(p)}\rangle.
$$

\begin{proposition}\label{pr:example}
	Let $\scrT$ be a $p$th ($p\ge 3$) order odeco tensor of dimension $d\times\cdots\times d$. If all its essential singular values are $\lambda$, then
	$$
	\min_{\bx^{(1)}\in \calS^{d-1}}\max_{\bx^{(2)}, \ldots, \bx^{(p)}\in \calS^{d-1}}\langle \scrT, \bx^{(1)}\otimes\cdots\otimes\bx^{(p)}\rangle={\lambda\over \sqrt{d}}.
	$$
\end{proposition}

Although straightforward generalizations of Weyl's bound to higher-order tensor do not hold, perturbation bounds in a similar spirit can still be established. More specifically, we have

\begin{theorem}\label{th:odeco-weyl}
	Let $\scrT$ and $\tilde{\scrT}$ be two $d_1\times\cdots\times d_p$ ($p>2$) odeco tensors with SVD:
	$$
	\scrT=[\{\lambda_k: 1\le k\le d_{\min}\}; \bU^{(1)},\ldots,\bU^{(p)}],
	$$
	and
	$$
	\tilde{\scrT}=[\{\tilde{\lambda}_k: 1\le k\le d_{\min}\}; \tilde{\bU}^{(1)},\ldots,\tilde{\bU}^{(p)}],
	$$
	respectively where $d_{\min}=\min\{d_1,\ldots, d_p\}$. There exist a numerical constant $1\le C\le 17$ and a permutation $\pi: [d_{\min}]\to [d_{\min}]$ such that for all $k=1,\ldots, d_{\min}$,
	\begin{equation}
		\label{eq:odecoweyl0}
		|\lambda_k-\tilde{\lambda}_{\pi(k)}|\le C\|\tilde{\scrT}-\scrT\|,
	\end{equation}
	and
	\begin{equation}
		\label{eq:odecodavis0}
		\max_{1\le q\le p}\sin\angle(\bu_k^{(q)}, \tilde{\bu}_{\pi(k)}^{(q)})\le {C\|\tilde{\scrT}-\scrT\|\over \lambda_k},
	\end{equation}
	with the convention that $1/0=+\infty$.
\end{theorem}

As discussed before, despite the similarity in appearance to the bounds for matrices, Theorem \ref{th:odeco-weyl} requires different proof techniques. Moreover, there are several intriguing differences between the bounds given in Theorem \ref{th:odeco-weyl} and classical ones for matrices.

First of all, we do not necessarily match the $k$th singular value/vector tuple $(\lambda_k, \bu_k^{(1)},\ldots, \bu_k^{(p)})$ of $\scrT$ with that of $\tilde{\scrT}$. This is because we do not restrict that the singular values $\lambda_k$s are distinct and sufficiently apart from each other, and hence the singular vectors corresponding to $\tilde{\lambda}_k$ are not necessarily close to those corresponding to $\lambda_k$. As a simple example, consider the following $2\times 2\times 2$ tensors:
$$
\scrT=(1+\delta)\bfe_1\otimes\bfe_1\otimes \bfe_1+(1-\delta)\bfe_2\otimes\bfe_2\otimes \bfe_2,
$$
and
$$
\tilde{\scrT}=(1-\delta)\bfe_1\otimes\bfe_1\otimes \bfe_1+(1+\delta)\bfe_2\otimes\bfe_2\otimes \bfe_2,
$$
where $\delta>0$ represents a small perturbation. Obviously, $\lambda_1=\tilde{\lambda}_1=1+\delta$ and $\lambda_2=\tilde{\lambda}_2=1-\delta$. But the correct way to study the effect of perturbation is to compare $(1+\delta)\bfe_1\otimes\bfe_1\otimes \bfe_1$ with $(1-\delta)\bfe_1\otimes\bfe_1\otimes \bfe_1$, and $(1-\delta)\bfe_2\otimes\bfe_2\otimes \bfe_2$ with $(1+\delta)\bfe_2\otimes\bfe_2\otimes \bfe_2$, and not the other way around. In other words, we want to pair $\lambda_1$ with $\tilde{\lambda}_2$, and $\lambda_2$ with $\tilde{\lambda}_1$.

Another notable difference is between the perturbation bound \eqref{eq:odecodavis0} for singular vectors and those from Wedin-Davis-Kahan $\sin\Theta$ theorems. The gap between singular values is absent in the bound \eqref{eq:odecodavis0}. This means that for higher-order odeco tensors, the perturbation affects the singular vectors separately. The perturbation bound \eqref{eq:odecodavis0} depends only on the amount of perturbation relative to their corresponding singular value.

For either \eqref{eq:odecoweyl0} or \eqref{eq:odecodavis0} to hold, we can take the constant $C=17$. It is plausible that this constant can be further improved. In general, for any such bound to hold, it is necessary that the constant $C\ge 1$ again by considering two rank-one tensors differing only in the nonzero singular value or in one of its corresponding singular vectors. Our next result shows that when the perturbation is sufficiently small, or for large enough singular values, we can indeed take $C=1$ or arbitrarily close to $1$.

\begin{theorem}\label{th:ortho-perturb}
	Let $\scrT$ and $\tilde{\scrT}$ be two $d_1\times\cdots\times d_p$ ($p>2$) odeco tensors with SVD:
	$$
	\scrT=[\{\lambda_k: 1\le k\le d_{\min}\}; \bU^{(1)},\ldots,\bU^{(p)}],
	$$
	and
	$$
	\tilde{\scrT}=[\{\tilde{\lambda}_k: 1\le k\le d_{\min}\}; \tilde{\bU}^{(1)},\ldots,\tilde{\bU}^{(p)}],
	$$
	respectively where $d_{\min}=\min\{d_1,\ldots, d_p\}$. There exists a permutation $\pi: [d_{\min}]\to [d_{\min}]$ such that for any $\eps>0$, 
	\begin{equation}
		\label{eq:odecoweyl}
		|\lambda_k-\tilde{\lambda}_{\pi(k)}|\le \|\tilde{\scrT}-\scrT\|,
	\end{equation}
	and
	\begin{equation}
		\label{eq:odecodavis}
		\max_{1\le q\le p}\sin\angle (\bu_k^{(q)},\tilde{\bu}_{\pi(k)}^{(q)})\le {(1+\eps)\|\tilde{\scrT}-\scrT\|\over \lambda_k}.
	\end{equation}
	provided that $\|\tilde{\scrT}-\scrT\|\le c_{\eps}\lambda_k$ for some constant $c_{\eps}>0$ depending on $\eps$ only.
\end{theorem}

The dependence of $c_\eps$ on $\eps$ can also be made explicit. In particular, we can take
$$
c_{\eps}=\min\{[1+2(1+\eps)]^{-1},\,h^{-1}(\eps/(1+\eps))\}
$$
where
\begin{align}\label{eq:cpdisp}
	h(x)&=(1+x)\left[1-\left({1-x}\over{1+x}\right)^{2}\right]^{1\over 2}+(1+\eps)x{(1+x)}. 
\end{align}
When considering infinitesimal perturbation in that $\|\tilde{\scrT}-\scrT\|=o(\lambda_k)$, we can express the bound \eqref{eq:odecodavis} for singular vectors as
$$
\max_{1\le q\le p}\sin\angle (\bu_k^{(q)},\tilde{\bu}_{\pi(k)}^{(q)})\le {\|\tilde{\scrT}-\scrT\|\over \lambda_k}+o\left({\|\tilde{\scrT}-\scrT\|\over \lambda_k}\right),
$$
which is more convenient for asymptotic analysis.

%
In general, if $\lambda_k=0$ for some $k<d_{\min}$, $\bu_k^{(q)}$s are not identifiable and therefore one cannot bound the effect of perturbation on the singular vectors in a meaningful way, as Theorems \ref{th:odeco-weyl} and \ref{th:ortho-perturb} also indicate. An exception is the case when $0$ is a singular value and simple, i.e., $\lambda_k>0$ for $k=1,\dots,d_{\min}-1$ and $\lambda_{d_{\min}}=0$. In this case, we can also derive nontrivial bounds for $(\bu_{d_{\min}}^{(1)},\ldots, \bu_{d_{\min}}^{(p)})$ since $\bu_{d_{\min}}^{(q)}$s are determined by $\bu_1^{(q)}, \ldots, \bu_{d_{\min}-1}^{(q)}$ for which the perturbation effect can be bounded appropriately. In particular, Theorems \ref{th:odeco-weyl} and \ref{th:ortho-perturb} provide perturbation bounds
$$
\max_{1\le q\le p}\sin\angle(\bu_k^{(q)}, \tilde{\bu}_{\pi(k)}^{(q)})\le {C\|\tilde{\scrT}-\scrT\|\over \lambda_k}
$$
for $1\le k\le d_{\min}-1$. By orthogonality of $\bU^{(q)}$ and $\tilde{\bU}^{(q)}$, this also means we have a perturbation bound for the last singular value-vector pair:
$$
\max_{1\le q\le p}
\sin\angle(\bu_{d_{\min}}^{(q)}, \tilde{\bu}_{\pi(d_{\min})}^{(q)})
\le \dfrac{C\|\tilde{\scrT}-\scrT\|}{\lambda_{d_{\min}-1}}.
$$

\subsection{Numerical Illustration}

To further illustrate these bounds, we carried out a couple of numerical experiments following the earlier work of \cite{mu2015successive}. In the first setting, we simulated two sets of i.i.d. random orthogonal matrices $\bU^{(q)}$, and $\bar{\bU}^{(q)}$ of dimension $20\times 10$, for $q=1,2,3$. We next generated $\hat{\bU}^{(q)}$ as the matrix with columns as $\hat{\bu}^{(q)}_i=\sqrt{1-\rho^2}\bu^{(q)}_i+\rho \bar{\bu}^{(q)}_i$, for $\rho=15/\lambda$ and $i=1,\dots,10$. Then we computed the orthogonal matrices $\tilde{\bU}^{(q)}$ through the polar decomposition of $\hat{\bU}^{(q)}=\tilde{\bU}^{(q)}\bP^{(q)}$. Finally we took the two odeco tensors as $\scrT=\lambda\displaystyle\sum_{i=1}^{10}\bu_i^{(1)}\otimes\bu_i^{(2)}\otimes\bu_i^{(3)}$ and $\tilde{\scrT}=\lambda\displaystyle\sum_{i=1}^{10}\tilde{\bu}_i^{(1)}\otimes\tilde{\bu}_i^{(2)}\otimes\tilde{\bu}_i^{(3)}$. We considered $\lambda=\omega\cdot d^{3/4}$ and vary $\omega$ over the 200 values in $\{1000,1000/2,\dots,1000/199,5\}$. Each point on the plot of Figure \ref{fig:rand_ortho} corresponds to one value of $\lambda$ and one random instance of $\scrT$ and $\tilde{\scrT}$. To fix ideas, on the Y axis we plot (in the notation of Theorem \ref{th:ortho-perturb}) the values $\max\sin\angle(\bu^{(q)}_k,\,\tilde{\bu}^{(q)}_{\pi(k)})$ where the maximum is over $1\le q\le 3$ and $1\le k\le 10$. On the X axis, we have $\|\tilde{\scrT}-\scrT\|/\lambda$ where the tensor spectral norm was evaluated based on 1000 random starts followed by power iteration.

\begin{figure}[htbp]
	\centering
	\begin{minipage}{0.5\textwidth}
		\centering
		\includegraphics[width=0.9\textwidth]{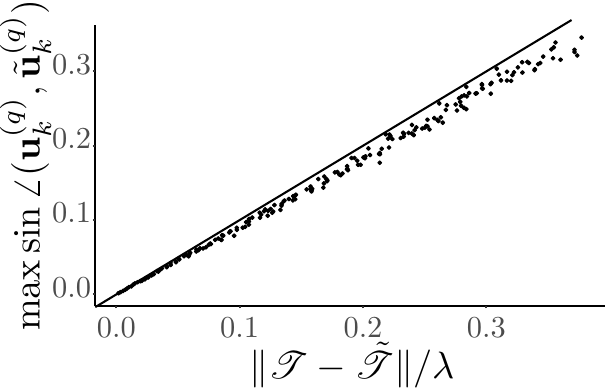}
		\caption{Random Orthogonal Tensors}\label{fig:rand_ortho}
	\end{minipage}\hfill
	\begin{minipage}{0.5\textwidth}
		\centering
		\includegraphics[width=0.9\textwidth]{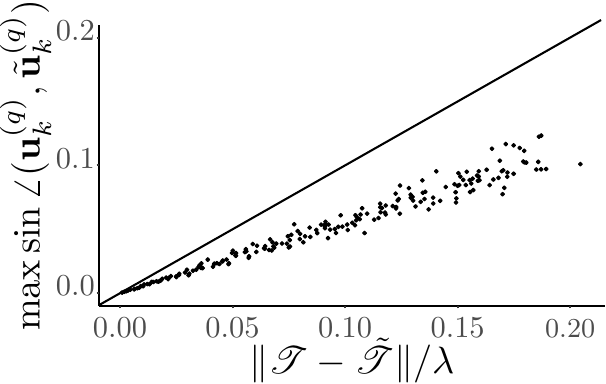} 
		\caption{Gaussian Errors}\label{fig:simul_gauss}
	\end{minipage}
\end{figure}

For reference, we add the $y=x$ line on the plot. It is evident that the maximum $\sin\angle$ distance between $\bu^{(q)}_k$ and  $\tilde{\bu}^{(q)}_{\pi(k)}$ is less than $\|\tilde{\scrT}-\scrT\|/\lambda $ on all instances, thus verifying Theorem \ref{th:ortho-perturb}. In Figure \ref{fig:rand_ortho}, the predicted bounds match almost exactly with the observed ones, showing the optimality of our bounds in this regime. 

In the second set of experiment, we took a $20\times 20\times 20$ tensor  $\scrT=\lambda\displaystyle\sum_{i=1}^{10}\be_i^{\otimes 3}+\scrE$, where the error tensor $\scrE$ consists of i.i.d. random errors $\eps_{ijk}\stackrel{iid}{\sim} N(0,1)$. We set $\lambda=\omega\cdot d^{3/4}$ and vary $\omega$ over the 200 values in $\{1000,1000/2,\dots,1000/199,5\}$. We computed an odeco approximation to $\scrX$ by random initialization followed by power iteration and successive deflation (as described in \cite{anand2014tensor} and \cite{mu2015successive}). Finally the LROAT algorithm of \cite{chen2009tensor} was used to obtain an odeco approximation $\tilde{\scrT}$.  As before, each point on the plot corresponds to one value of $\lambda$ and one random instance of $\scrE$. This numerical study can be directly compared to the simulation studies from \cite{mu2015successive}, where an algorithm-dependent perturbation bound was derived in a similar setting. Our upper bound is significantly tighter, and appears to be optimal when the perturbation is small.

\subsection{Perturbation of Nonessential Singular Vectors}
Thus far, we have focused on the essential singular values and vectors of odeco tensors. In deriving their perturbation bounds, we actually established more precise characterization of the perturbation effect on $\bu^{(q)}_k$s. See \eqref{eq:secordpert}. It turns out that we can leverage such a characterization to develop perturbation bounds for general real singular vectors of an odeco tensor. 

Denote by $r$ the rank of $\scrT$, or equivalently the number of nonzero $\lambda_k$. Write
$$
\bM^{(q)}=\left[(\tilde{\scrT}-\scrT)\times_{s\neq q}\bu_1^{(s)}\,\dots\,
(\tilde{\scrT}-\scrT)\times_{s\neq q}\bu_{d_r}^{(s)}
\right]
$$

\begin{theorem}\label{th:allsingpert} 
	Assume that
	$$
	\max_{1\le q\le p}\left\|\bM^{(q)}\right\|\le C_1\|\tilde{\scrT}-\scrT\|\qquad and\qquad 
	\|\tilde{\scrT}-\scrT\|\le C_2\lambda_rr^{-1/2(p-2)}.
	$$
	for some constants $C_1,C_2>0$. If $(\lambda;\bv^{(1)},\dots,\bv^{(p)})$ is a singular value-vector tuple of $\scrT$, then there exists a singular value/vector tuple $(\tilde{\lambda};\tilde{\bv}^{(1)},\dots,\tilde{\bv}^{(p)})$ of $\tilde{\scrT}$ such that 
	$$
	\max_{1\le q\le p}\|\bv^{(q)}-\tilde{\bv}^{(q)}\|\le \dfrac{C\|\tilde{\scrT}-\scrT\|}{\lambda_{\min}^*}
	$$
	where $\lambda_{\min}^*=\min\{\lambda_k:|\langle\bv^{(1)},\bu^{(1)}_k\rangle|>0\}$ and $C$ is a constant depending on $C_1,C_2$ and $p$ only. Here we use the convention that $\lambda_{\min}^*=0$ if  $\{\lambda_k:|\langle\bv^{(1)},\bu^{(1)}_k\rangle|>0\}=\emptyset$ and $1/0=+\infty$. 
	
	Furthermore, if $\lambda>0$, then
	$$
	\abs*{\tilde{\lambda}-\lambda}\le \dfrac{C\lambda\|\tilde{\scrT}-\scrT\|}{\lambda_{\min}^*}.
	$$
\end{theorem}

\section{Nearly Orthogonal Tensors}
The perturbation bounds for singular values and vectors derived in the previous section have a direct generalization to a larger class of tensors that are close to being odeco. In particular, let $\scrT_1$ and $\scrT_2$ be any two odeco approximations of a tensor $\scrX$. By triangle inequality we have
$$
\|\scrT_1-\scrT_2\|\le \|\scrT_1-\scrA\|+\|\scrT_2-\scrA\|.
$$
Now applying Theorems \ref{th:odeco-weyl} and \ref{th:ortho-perturb} one obtains perturbation bounds for the singular values and singular vectors of $\scrT_1$ and $\scrT_2$. Note that these bounds depend on the quality of approximation $\|\scrT_i-\scrA\|$. It is natural that a tensor $\scrA$ close to being odeco can be approximated better in this fashion. Finally, we do not require any optimality property of the odeco approximations. The perturbation bounds hold for any such approximation $\scrT_1$ and $\scrT_2$, although the bounds are sharper for better approximations. We now illustrate useful applications of these bounds in several more concrete settings. 

\subsection{Perturbation Bounds for Incoherent Tensors} 

Consider
\begin{equation}
	\label{eq:incoherent}
	\scrX=\sum_{k=1}^{r} \eta_k \ba_k^{(1)}\otimes\cdots\ba_k^{(p)}\quad\text{and}
	\quad
	\tilde{\scrX}=\sum_{k=1}^{\tilde{r}} \tilde{\eta}_k \tilde{\ba}_k^{(1)}\otimes\cdots\tilde{\ba}_k^{(p)}
\end{equation}
where $\eta_1\ge \ldots\ge\eta_{r}> 0$ and $\tilde{\eta}_1\ge\ldots\ge \tilde{\eta}_{\tilde{r}}>0$. Different from odeco tensors, the unit vectors $\ba_k^{(q)}$s and $\tilde{\ba}_k^{(q)}$ in \eqref{eq:incoherent} are not required to be orthonormal but assumed to be close to being orthonormal. More specifically, we shall assume that $\bA^{(q)}$s and $\tilde{\bA}^{(q)}$ satisfy the isometry condition
\begin{equation}
	\label{eq:isometry}
	1-\delta\le \min\{\lambda_{\min}(\bA^{(q)}),\,\lambda_{\min}(\tilde{\bA}^{(q)})\}\le \max\{\lambda_{\max}(\bA^{(q)}), \lambda_{\max}(\tilde{\bA}^{(q)})\}\le 1+\delta
\end{equation}
for all $q=1,\dots,p$ for some $0\le \delta<1$, where $\bA^{(q)}=[\ba_1^{(q)},\ldots,\ba_{r}^{(q)}]$, $\tilde{\bA}^{(q)}=[\tilde{\ba}_1^{(q)},\ldots,\tilde{\ba}_{\tilde{r}}^{(q)}]$ and $\lambda_{\min}(\cdot)$ and $\lambda_{\max}(\cdot)$ evaluate the smallest and largest singular values, respectively, of a matrix. Clearly $\delta=0$ if $\bA^{(q)},\,\tilde{\bA}^{(q)}$ are orthonormal so that $\delta$ measures the incoherence of its column vectors. A canonical example of incoherent tensors arises in a probabilistic setting: let $\ba_1^{(q)},\ldots, \ba_r^{(q)}$ be independently and uniformly sampled from the unit sphere; then it is not hard to see that $\delta=O_p(\sqrt{r/d_q})$.

In light of Kruskal's Theorem  \citep{kruskal1977three}, the decompositions in  \eqref{eq:incoherent} are essentially unique and therefore $\scrX$ and  $\tilde{\scrX}$ cannot be odeco unless $\delta=0$. However, $\scrX$, $\tilde{\scrX}$ are close to being odeco when $\delta$ is small. More specifically, let $\bA^{(q)}=\bU^{(q)}\bP^{(q)}$ and $\tilde{\bA}^{(q)}=\tilde{\bU}^{(q)}\tilde{\bP}^{(q)}$ be their polar decompositions, and
\begin{equation}
	\label{eq:defodecX}
	\scrT=\sum_{k=1}^{r} \eta_k \bu_k^{(1)}\otimes\cdots\otimes\bu_k^{(p)}\quad\text{and }\quad 
	\tilde{\scrT}=\sum_{k=1}^{\tilde{r}} \tilde{\eta}_k \tilde{\bu}_k^{(1)}\otimes\cdots\otimes\tilde{\bu}_k^{(p)}.
\end{equation}
It is clear that $\scrT$ and $\tilde{\scrT}$ are odeco and moreover, we can show that

\begin{theorem}
	\label{th:incoherent}
	Let $\scrX$, $\tilde{\scrX}$ be defined by \eqref{eq:incoherent} with the unit vectors $\ba_k^{(q)}$s and $\tilde{\ba}_k^{(q)}$s obeying \eqref{eq:isometry}, and $\scrT$, $\tilde{\scrT}$ by \eqref{eq:defodecX}. Then
	$$
	\|\scrT-\scrX\|\le (p+1)\delta\eta_1,\quad\text{and}\quad \|\tilde{\scrT}-\tilde{\scrX}\|\le (p+1)\delta\tilde{\eta}_1
	$$
	and
	$$
	\max\left\{\max_{1\le q\le p}\sin\angle (\ba_k^{(q)},\bu_k^{(q)}),\,\max_{1\le q\le p}\sin\angle (\tilde{\ba}_k^{(q)},\tilde{\bu}_k^{(q)})\right\}\le \delta/\sqrt{2}.
	$$
\end{theorem}

Theorem \ref{th:incoherent} reveals that there exist odeco approximations $\scrT$ and $\tilde{\scrT}$ such that  $\|\scrX-\scrT\|\le (p+1)\delta\eta_1$ and $\|\tilde{\scrX}-\tilde{\scrT}\|\le (p+1)\delta\tilde{\eta}_1$ respectively. This is conjunction with Theorem \ref{th:odeco-weyl} implies a perturbation bound for the components of $\scrX$ and $\tilde{\scrX}$.

\begin{corollary}
	\label{co:incoherent}
	Let $\scrX$ and $\tilde{\scrX}$ be defined by \eqref{eq:incoherent} with the unit vectors $\ba_k^{(q)}$s and $\tilde{\ba}_k^{(q)}$ obeying \eqref{eq:isometry}. Then there exist a numerical constant $C>0$ and a permutation $\pi: [d_{\min}]\to [d_{\min}]$ such that for any $1\le k\le r$,
	$$
	|\eta_k-\tilde{\eta}_{\pi(k)}|\le C[(p+1)\delta(\eta_1+\tilde{\eta}_1)+\|\scrX-\tilde{\scrX}\|]
	$$
	and
	$$
	\max_{1\le q\le p}\sin\angle (\ba_k^{(q)},\tilde{\ba}_{\pi(k)}^{(q)})\le C\{(p+1)\delta(\eta_1+\tilde{\eta}_1)+\|\scrX-\tilde{\scrX}\|+\delta\}/\eta_k.
	$$
\end{corollary}

We want to point out that Corollary \ref{co:incoherent} can also be viewed as a ``robust'' version of Theorem \ref{th:odeco-weyl} as the latter can be viewed as a special case of the former when $\delta=0$. 

\subsection{Additive Perturbation of Odeco Tensor}

The perturbation bounds we derived are fairly general and can be applied to various problems in statistics and machine learning. For example, in a typical spectral learning scenario, we observe $\scrX$, which is  $\scrT$ ``contaminated'' by an additive perturbation $\scrE$ and want to infer from $\scrX$ the $\bu_k^{(q)}$s. See, e.g., \cite{janzamin2019spectral}. In general, $\scrX=\scrT+\scrE$ is no longer odeco, and we may not be able to define its SVD in the same fashion as \eqref{eq:svd}. However, when $\|\scrE\|\le \eps$,  any $\eps$-odeco approximation of $\scrX$ is necessarily close to $\scrT$ as well and its singular values and vectors to those of $\scrT$. More precisely, we have the following result as an immediate consequence of Theorem \ref{th:odeco-weyl}.

\begin{corollary}
	\label{co:nearly}
	Let
	$$
	\scrT=[\{\lambda_k: 1\le k\le d_{\min}\}; \bU^{(1)},\ldots,\bU^{(p)}],
	$$
	be an odeco tensor ($p\ge 3$), and
	$$
	\tilde{\scrT}=[\{\tilde{\lambda}_k: 1\le k\le d\}; \tilde{\bU}^{(1)},\ldots,\tilde{\bU}^{(p)}]
	$$
	be any odeco approximation to $\scrX:=\scrT+\scrE$. Then there are a numerical constant $C\ge 1$ and a permutation $\pi: [d]\to [d]$ such that
	\begin{equation}
		\label{eq:coweyl}
		|\lambda_k-\tilde{\lambda}_{\pi(k)}|\le C(\|\tilde{\scrT}-\scrX\|+\|\scrE\|),
	\end{equation}
	and
	\begin{equation}
		\label{eq:codavis}
		\max_{1\le q\le p}\sin\angle (\bu_k^{(q)},\tilde{\bu}_{\pi(k)}^{(q)})\le {C(\|\tilde{\scrT}-\scrX\|+\|\scrE\|)\over \lambda_k}
	\end{equation}
	for all $k=1,\ldots, d_{\min}$.
\end{corollary}

While Corollary \ref{co:nearly} holds for any odeco approximation to $\scrX$, it is oftentimes of interest to ``estimate'' the singular values and vectors of $\scrT$ using a ``good'' odeco approximation. In particular, one may consider the best odeco approximation:
$$
\tilde{\scrT}^{\rm best}:=\inf_{\scrA{\rm \ is\ odeco}}\|\scrX-\scrA\|.
$$
It is clear that
$$
\|\tilde{\scrT}^{\rm best}-\scrX\|\le \|\scrT-\scrX\|=\|\scrE\|,
$$
where $\tilde{\scrT}^{\rm best}=[\{\tilde{\lambda}_k: 1\le k\le d\}; \tilde{\bU}^{(1)},\ldots,\tilde{\bU}^{(p)}]$. The bounds \eqref{eq:coweyl} and \eqref{eq:codavis} now become:
\begin{equation}
	\label{eq:coweyl1}
	|\lambda_k-\tilde{\lambda}_{\pi(k)}|\le C\|\scrE\|,
\end{equation}
and
\begin{equation}
	\label{eq:codavis1}
	\max_{1\le q\le p}\sin\angle (\bu_k^{(q)},\tilde{\bu}_{\pi(k)}^{(q)})\le {C\|\scrE\|\over \lambda_k}.
\end{equation}
Indeed \eqref{eq:coweyl1} and \eqref{eq:codavis1} continue to hold for any odeco approximation $\tilde{\scrT}$ obeying
\begin{equation}
	\label{eq:odecoappr}
	\|\tilde{\scrT}-\scrX\|\lessim \|\scrE\|.
\end{equation}

It is worth noting that computing an odeco approximation that satisfies \eqref{eq:odecoappr} is not always straightforward.  In fact, development of efficient algorithms that can produce ``good'' odeco approximation to a nearly odeco tensor is an active research area with fervent interest. A flurry of recent works suggest that finding a $\scrX^{\rm odeco}$ satisfying \eqref{eq:odecoappr} is feasible at least when $\|\scrE\|$ is sufficiently small. Interested readers are referred to \cite{anand2014tensor, mu2015successive,mu2017greedy, belkin2018eigenvectors} and references therein for more detailed discussions regarding this aspect.

\subsection{High Dimensional Tensor SVD}
\label{sec:app}
A particularly common type of perturbation $\scrE$ is a noisy tensor whose entries are independent standard normal random variables, in particular when $d_j$s are large. The so-called tensor SVD problem has been studied earlier by \cite{richard2014statistical,liu2017characterizing,zhang2018tensor} among others, and is among the most commonly used methods to reduce the dimensionality of the data, and oftentimes serves as a useful first step to capture the essential features in the data for downstream analysis.

As noted before, a natural estimate of $\scrT$ is the best odeco approximation to $\scrX$:
\begin{equation}
	\label{eq:defhatTsvd}
	\hat{\scrT}=[\{\hat{\lambda}_k: 1\le k\le d_{\min}\}; \hat{\bU}^{(1)},\ldots,\hat{\bU}^{(p)}]:=\argmin_{\scrA {\rm \ is\ odeco}}\|\scrX-\scrA\|.
\end{equation}
A standard argument yields $\|\scrE\|=O_p(\sqrt{d_1+\cdots+d_p})$. See, e.g., \cite{raskutti2019convex}. Together with Corollary \ref{co:nearly}, this implies that there exists a permutation $\pi: [d_{\min}]\to [d_{\min}]$ so that
\begin{equation}
	\label{eq:svdbd1}
	\EE \max_{1\le k\le d_{\min}}|\hat{\lambda}_{\pi(k)}-\lambda_k|\le C\cdot\sqrt{d_1+\cdots+d_p},
\end{equation}
and
\begin{equation}
	\label{eq:svdbd2}
	\EE\max_{1\le q\le p} \sin\angle(\bu_k^{(q)},\hat{\bu}_{\pi(k)}^{(q)})\le C\cdot \min\left\{{\sqrt{d_1+\ldots+d_p}\over \lambda_k},1\right\},
\end{equation}
for any $k=1,\ldots,d_{\min}$.

Bounds similar to those given by \eqref{eq:svdbd2} are known when $\scrT$ is of rank one, that is, $\lambda_2=\cdots=\lambda_{d_{\min}}=0$. See, e.g., \cite{richard2014statistical}. \eqref{eq:svdbd2} indicates that the same bounds hold uniformly over all singular values and vectors of an odeco tensor. In other words, we can estimate any singular value and vectors of $\scrT$ at the same rate as if all other singular values are zero or equivalently as in the rank one case. This also draws contrast with the setting considered by \cite{zhang2018tensor}. Generalizing the rank-one model of \cite{richard2014statistical}, \cite{zhang2018tensor} studies efficient estimation strategies of $\scrT$ and its decomposition when it is of low multilinear ranks. Their analysis requires that $\scrT$ is nearly cubic, e.g., $d_1\asymp d_2\asymp d_3$, and the ranks are of an order up to $d_1^{1/2}$ among other conditions. Odeco tensors have more innate structure and consequently, as \eqref{eq:svdbd1} and \eqref{eq:svdbd2} show, if $\scrT$ is odeco, its shape and rank are irrelevant for estimating its singular values and vectors.

Note that if instead we use the perturbation bounds derived from matricization, the factor $\sqrt{d_1+\cdots+d_p}$ on the righthand side of \eqref{eq:svdbd1} and \eqref{eq:svdbd2} becomes $\sqrt{d_1\times\cdots\times d_{p-1}}$ which can be significantly larger. Indeed, both bounds \eqref{eq:svdbd1} and \eqref{eq:svdbd2} can be shown to be minimax optimal in that no other estimates of the singular vectors or values based upon $\scrX$ could attain a faster rate of convergence, and therefore characterize the exact effect of dimensionality on our ability to infer $\lambda_k$s or $\bu_k$s.

\begin{theorem}
	\label{th:tpcalower}
	Consider the tensor SVD model $\scrX=\scrT+\scrE$ where
	$$
	\scrT=[\{\lambda_k: 1\le k\le d_{\min}\}; \bU^{(1)},\ldots,\bU^{(p)}]
	$$
	is odeco and $\scrE$ has independent standard normal entries. Then there exists a constant $c>0$ such that
	\begin{equation}
		\label{eq:svdbd3}
		\inf_{\tilde{\lambda}_k}\sup_{\bu_k^{(q)}\in \calS^{d_q-1}: 1\le q\le p}\EE |\tilde{\lambda}_k-\lambda_k|\ge c\cdot\sqrt{d_1+\cdots+d_p},
	\end{equation}
	and
	\begin{equation}
		\label{eq:svdbd4}
		\inf_{\tilde{\bu}_k^{(1)},\ldots,\tilde{\bu}_k^{(p)}}\sup_{\bu_k^{(q)}\in \calS^{d_q-1}: 1\le q\le p}\EE\max_{1\le q\le p} \sin\angle(\bu_k^{(q)},\tilde{\bu}_k^{(q)})\ge c\cdot\min\left\{{\sqrt{d_1+\ldots+d_p}\over \lambda_k},1\right\},
	\end{equation}
	where the infimum in \eqref{eq:svdbd3} and \eqref{eq:svdbd4} is taken over all estimates of the form $\hat{\scrT}=\sum_{k=1}^{d_{\min}}\hat{\lambda}_k\tilde{\bu}_k^{(1)}\otimes\dots\otimes\tilde{\bu}_k^{(p)}$, (not necessarily odeco) based on observing $\scrX$.
\end{theorem}

Theorem \ref{th:tpcalower} again confirms that our perturbation bounds is optimal at least up to a numerical constant.
\section{Proofs}

\begin{proof}[Proof of Proposition \ref{pr:example}]
	For brevity, we shall assume that $\lambda=1$. For any $\bx^{(1)}\in \calS^{d-1}$, there exists $\bc=(c_1,\ldots, c_d)^\top \in \calS^{d-1}$ such that
	$$
	\bx^{(1)}=c_1\bu_1^{(1)}+\ldots+c_d\bu_d^{(1)}.
	$$
	Therefore
	\begin{eqnarray*}
		&&\min_{\bx^{(1)}\in \calS^{d-1}}\max_{\bx^{(2)},\ldots, \bx^{(p)}\in \calS^{d-1}}\langle \scrT, \bx^{(1)}\otimes\cdots\otimes\bx^{(p)}\rangle\\
		&=&\min_{\bc\in \calS^{d-1}}\max_{\bx^{(2)},\ldots, \bx^{(p)}\in \calS^{d-1}}\langle \scrT, (c_1\bu_1^{(1)}+\ldots+c_d\bu_d^{(1)})\otimes\bx^{(2)}\otimes\cdots\otimes\bx^{(p)}\rangle\\
		&=&\min_{\bc\in \calS^{d-1}}\max_{\bx^{(2)},\ldots, \bx^{(p)}\in \calS^{d-1}}\langle \sum_{k=1}^d c_k\bu_k^{(2)}\otimes\cdots\otimes\bu_k^{(p)}, \bx^{(2)}\otimes\cdots\otimes\bx^{(p)}\rangle\\
		&=&\min_{\bc\in \calS^{d-1}}\left\|\sum_{k=1}^d c_k\bu_k^{(2)}\otimes\cdots\otimes\bu_k^{(p)}\right\|.
	\end{eqnarray*}
	Note that since 
	$$
	\sum_{k=1}^d c_k\bu_k^{(2)}\otimes\cdots\otimes\bu_k^{(p)}
	$$
	is a $(p-1)$th order odeco tensor, we get
	$$
	\left\|\sum_{k=1}^d c_k\bu_k^{(2)}\otimes\cdots\otimes\bu_k^{(p)}\right\|=\max_{1\le k\le d}c_k,
	$$
	so that
	$$
	\min_{\bx^{(1)}\in \calS^{d-1}}\max_{\bx^{(2)},\ldots, \bx^{(p)}\in \calS^{d-1}}\langle \scrT, \bx^{(1)}\otimes\cdots\otimes\bx^{(p)}\rangle=\min_{\bc\in \calS^{d-1}}\max_{1\le k\le d}c_k={1\over \sqrt{d}}
	$$
	The proof is now completed.
\end{proof}
\vskip 25pt

\begin{proof}[Proof of Theorem \ref{th:odeco-weyl}]  In fact,  Theorem  \ref{th:odeco-weyl} follows from  Theorem \ref{th:ortho-perturb}. We shall now describe how we may proceed to prove Theorem \ref{th:odeco-weyl} in light of   Theorem \ref{th:ortho-perturb}. Indeed, note that Theorem \ref{th:ortho-perturb} already shows that \eqref{eq:odecoweyl0} and \eqref{eq:odecodavis0} hold for any $\|\scrT-\tilde{\scrT}\|\le c_\eps\lambda_k$ with an appropriate choice of constant $C=1+\eps$. When $\|\scrT-\tilde{\scrT}\|>c_{\eps}\lambda_k$, then for any $k'\in[d]$,
	$$
	\max_{1\le q\le p}\sin\angle(\bu_k^{(q)},\tilde{\bu}_{k'}^{(q)})\le 1\le {\|\scrT-\tilde{\scrT}\|\over c_\eps\lambda_k},
	$$
	so that \eqref{eq:odecodavis0} holds with $C=\max\{(1+\eps),1/c_{\eps}\}$.
	
	A careful inspection of the proof to Theorem \ref{th:ortho-perturb} also shows that \eqref{eq:odecoweyl0} holds with $C=1$ not only for any $\lambda_k\ge  \|\scrT-\tilde{\scrT}\|/c_\eps$ but also for any $\tilde{\lambda}_k\ge  \|\scrT-\tilde{\scrT}\|/c_\eps$. On the other hand, for any $\lambda_k< \|\scrT-\tilde{\scrT}\|/c_\eps$ and $\lambda_{k'}< \|\scrT-\tilde{\scrT}\|/c_\eps$, we must have
	$$
	|\lambda_k-\tilde{\lambda}_k|<\|\scrT-\tilde{\scrT}\|/c_\eps,
	$$
	which again suggests that \eqref{eq:odecoweyl0} holds with $C=1/c_\eps$. Optimizing over all possible $\eps$ we have that $\max\{1+\eps,1/c_{\eps}\}$ is minimized at $\eps=2.94$ for an objective value of $16.48$, so that we can take $C=17$.
\end{proof}
\vskip 25pt

\begin{proof}[Proof of Theorem \ref{th:ortho-perturb}]
	We shall derive the stronger statement:
	\begin{align*}\label{eq:secordpert}
		\max_{1\le q\le p}&\sin\angle\left(\tilde{\bu}_{\pi(k)}^{(q)},\,\bu_{k}^{(q)}
		+\dfrac{1}{\lambda_{k}}(\tilde{\scrT}-\scrT)\times_{s\neq q}
		\bu_{k}^{(s)}\right)\\
		\le &
		\left(2+{\|\tilde{\scrT}-\scrT\|\over\lambda_k}\right)
		\left(\dfrac{(1+\eps)\|\tilde{\scrT}-\scrT\|}{\lambda_{k}}\right)^{p-1}\numberthis
	\end{align*}
	for $k\in[d_{\min}]$.
	
	The proof is fairly involved and we begin by giving a short summary of the main challenges and ideas. The proof proceeds by induction over $k$. For the basic case when $k=1$, we can derive the bounds for $\tilde{\lambda}_1$ and $\tilde{\bu}_1$ using the variational characterization of $\lambda_1$ and $\tilde{\lambda}_1$. Special attention is needed to deal with the case when the best rank one approximation of $\scrT$ is not unique, or equivalently when $\lambda_1$ is not simple. In this case it is crucial to identify the right singular vectors of $\tilde{\scrT}$ to be matched with $\bu_1^{(q)}$s. A heuristic argument why this is possible for $p>2$ can be illustrated in the case when $\scrT=\sum_{i=1}^d \bfe_i^{\otimes p}$. When $p=2$, any $\bx\in \calS^{d-1}$ satisfies $\langle \scrT, \bx\otimes \bx\rangle=1$ so we cannot recover the singular vectors, even without perturbation. When $p>2$, however,
	$$
	\max_{\bx\in \calS^{d-1}} \langle \scrT, \bx\otimes\cdots\otimes \bx\rangle=\sum_{i=1}^d x_i^p\le \max_{1\le i\le d} |x_i|^{p-2}
	$$
	so that the maximum is attained only when $x=\bfe_i$ for some $i$. 
	
	The case when $k>1$ is more delicate where we need to make use of the fact that we have matched all leading $k-1$ singular values and vectors of $\scrT$ and $\tilde{\scrT}$. Due to the lack of Courant-Fischer-Weyl min-max principle for higher order tensors, we can only resort to, again, the variational characterization of $\lambda_k$ and $\tilde{\lambda}_k$. Note that we could proceed in an identical fashion as the basic case if $\tilde{\bu}_l^{(q)}=\bu_l^{(q)}$ for $l<k$. Of course this is not the case. Nonetheless we do have bounds for the perturbation of  $\tilde{\lambda}_1,\ldots, \tilde{\lambda}_{k-1}$ and $\tilde{\bu}_1^{(q)},\ldots, \tilde{\bu}_{k-1}^{(q)}$ in the form of \eqref{eq:secordpert}. By carefully leveraging these perturbation bounds, along with the fact that $\bu_k^{(q)}$ and $\tilde{\bu}_k^{(q)}$ must be orthogonal to $\{\bu_1^{(q)},\ldots, \bu_{k-1}^{(q)}\}$ and $\{\tilde{\bu}_1^{(q)},\ldots, \tilde{\bu}_{k-1}^{(q)}\}$, we can derive the desired perturbation bounds for $\tilde{\lambda}_k$ and $\tilde{\bu}^{(q)}_k$ and therefore conclude the induction.
	
	We are now in position to present the detailed proof. For notational convenience, we shall assume that $d_1=\cdots=d_p=:d$. The general proof follows by identical steps, with $d$ replaced everywhere by $d_{\min}$ and some equality signs replaced by ``less than or equal to" signs. In the body of the proof, we indicate which steps have such a change. As noted before, the proof proceeds by induction. To this end, we first consider the basic case when $k=1$.
	
	\paragraph{\textbf{Basic case}} Recall that
	$$
	\lambda_1=\max_{\ba^{(q)}\in \calS^{d-1}: q=1,\ldots, p} \langle\scrT, \ba^{(1)}\otimes\cdots\otimes\ba^{(p)}\rangle
	$$
	and
	$$
	\tilde{\lambda}_1=\max_{\ba^{(q)}\in \calS^{d-1}: q=1,\ldots, p} \langle\tilde{\scrT}, \ba^{(1)}\otimes\cdots\otimes\ba^{(p)}\rangle
	$$
	We consider separately the cases when $\tilde{\lambda}_1\le \lambda_1$ and $\tilde{\lambda}_1> \lambda_1$. 
	\bigskip
	\paragraph{\textbf{Basic case (a): $\tilde{\lambda}_1\le \lambda_1$}} Observe that
	\begin{eqnarray*}
		\lambda_1&=&\langle \scrT,\bu_1^{(1)}\otimes\cdots\otimes \bu_1^{(p)}\rangle\\
		&\le&\langle \tilde{\scrT},\bu_1^{(1)}\otimes\cdots\otimes\bu_1^{(p)}\rangle+
		|\langle \tilde{\scrT}-\scrT,\bu_1^{(1)}\otimes\cdots\otimes\bu_1^{(p)}\rangle|
		\\
		&=&\sum_{k=1}^d \tilde{\lambda}_k\prod_{q=1}^p \langle \bu_1^{(q)},\tilde{\bu}_k^{(q)}\rangle+
		|\langle \tilde{\scrT}-\scrT,\bu_1^{(1)}\otimes\cdots\otimes\bu_1^{(p)}\rangle|.
	\end{eqnarray*}
	The first term can be further bounded by
	\begin{eqnarray*}
		&&\sum_{k=1}^d \tilde{\lambda}_k\prod_{q=1}^p \langle \bu_1^{(q)},\tilde{\bu}_k^{(q)}\rangle\\
		&\le&\max_{1\le k\le d}\left\{\tilde{\lambda}_k\prod_{q=1}^p |\langle \bu_1^{(q)},\tilde{\bu}_k^{(q)}\rangle|^{(p-2)/p}\right\}\times\left(\sum_{k=1}^d\prod_{q=1}^p |\langle \bu_1^{(q)},\tilde{\bu}_k^{(q)}\rangle|^{2/p}\right)\\
		&\le&\max_{1\le k\le d}\left\{\tilde{\lambda}_k\prod_{q=1}^p |\langle \bu_1^{(q)},\tilde{\bu}_k^{(q)}\rangle|^{(p-2)/p}\right\}\times \left(\prod_{q=1}^p\left(\sum_{k=1}^d|\langle \bu_1^{(q)},\tilde{\bu}_k^{(q)}\rangle|^2\right)\right)^{1/p}\\
		\footnotemark &\le &\max_{1\le k\le d}\left\{\tilde{\lambda}_k\prod_{q=1}^p |\langle \bu_1^{(q)},\tilde{\bu}_k^{(q)}\rangle|^{(p-2)/p}\right\},
	\end{eqnarray*}
	
	\footnotetext{\label{note1}This line holds with equality when $d_1=\dots=d_p$ but is a ``$\le$" in general.}
	
	\noindent where the second inequality follows from Holder's inequality. Denote by $\pi(1)$ the index that maximizes the rightmost hand side. When there are more than one maximizers, we take $\pi(1)$ to be an arbitrary maximizing index. Then
	$$
	\lambda_1\le \tilde{\lambda}_{\pi(1)}+
	|\langle \tilde{\scrT}-\scrT,\bu_1^{(1)}\otimes\cdots\otimes\bu_1^{(p)}\rangle|,
	$$
	which, together with the fact that $\tilde{\lambda}_{\pi(1)}\le \tilde{\lambda}_1\le\lambda_1$, implies (by analogous calculations) that
	$$
	|\lambda_1- \tilde{\lambda}_{\pi(1)}|\le
	\max\{|\langle \tilde{\scrT}-\scrT,\bu_1^{(1)}\otimes\cdots\otimes\bu_1^{(p)}\rangle|,
	\,
	|\langle \tilde{\scrT}-\scrT,\tilde{\bu}_{\pi(1)}^{(1)}\otimes\cdots
	\otimes\tilde{\bu}_{\pi(1)}^{(p)}\rangle|
	\}.
	$$
	In addition,
	$$\begin{aligned}
		\lambda_1&\le \tilde{\lambda}_{\pi(1)}\prod_{q=1}^p |\langle \bu_1^{(q)},\tilde{\bu}_{\pi(1)}^{(q)}\rangle|^{(p-2)/p}+
		|\langle \tilde{\scrT}-\scrT,\bu_1^{(1)}\otimes\cdots\otimes\bu_1^{(p)}\rangle|
		\\&\le \lambda_1\prod_{q=1}^p |\langle \bu_1^{(q)},\tilde{\bu}_{\pi(1)}^{(q)}\rangle|^{(p-2)/p}+
		|\langle \tilde{\scrT}-\scrT,\bu_1^{(1)}\otimes\cdots\otimes\bu_1^{(p)}\rangle|.
	\end{aligned}
	$$
	Thus,
	$$
	\left(\prod_{q=1}^p|\langle \bu_1^{(q)},\tilde{\bu}_{\pi(1)}^{(q)}\rangle|\right)^{1/p}\ge (1-\lambda_1^{-1}
	|\langle \tilde{\scrT}-\scrT,\bu_1^{(1)}\otimes\cdots\otimes\bu_1^{(p)}\rangle|)^{1/(p-2)}
	$$
	for all $q=1,\ldots, p$. Now recall that
	$$
	\sin^2\angle (\bu_1^{(q)},\tilde{\bu}_{\pi(1)}^{(q)})=\langle \bu_1^{(q)}-\langle \bu_1^{(q)},\tilde{\bu}_{\pi(1)}^{(q)}\rangle\tilde{\bu}_{\pi(1)}^{(q)},\bu_1^{(q)}\rangle.
	$$
	We get
	
	\begin{equation}\label{eq:prebd}
		\begin{split}
			\prod_{q=1}^p\sin^2\angle (\bu_1^{(q)},\tilde{\bu}_{\pi(1)}^{(q)})
			=&\prod_{q=1}^p(1-\langle \bu_1^{(q)},\tilde{\bu}_{\pi(1)}^{(q)}\rangle^2) \\
			\le & \left(1-\frac{1}{p}\sum_{q=1}^p\langle \bu_1^{(q)},\tilde{\bu}_{\pi(1)}^{(q)}\rangle^2\right)^p\\
			\le & \left(1-\left(\prod_{q=1}^p|\langle \bu_1^{(q)},\tilde{\bu}_{\pi(1)}^{(q)}\rangle|\right)^{2/p}\right)^p\\
			\le & \left(1- (1-\lambda_1^{-1}\|\tilde{\scrT}-\scrT\|)^{2/(p-2)}\right)^p.	
		\end{split}
	\end{equation}
	In the above we use AM-GM inequality to get the first and second inequalities.
	We shall now use this to derive a sharper bound for the lefthand side.
	
	Note that
	$$
	\scrT(\bI,\bu_{1}^{(2)},\ldots,\bu_{1}^{(p)})
	=\lambda_1\bu_{1}^{(1)}.
	$$
	
	Thus, for any unit vector $\bv\perp\tilde{\bu}^{(1)}_{\pi(1)}$ we get
	\begin{align*}
		\lambda_{1}\langle \bv,\,\bu_{1}^{(1)} \rangle
		=& \scrT(\bv,\bu_{1}^{(2)},\ldots,\bu_{1}^{(p)})\\
		=& (\scrT-\tilde{\scrT})(\bv,\bu_{1}^{(2)},\ldots,\bu_{1}^{(p)})
		+\tilde{\scrT}(\bv,\bu_{1}^{(2)},\ldots,\bu_{1}^{(p)}).
	\end{align*}
	
	
	The second term on the rightmost hand side can be further bounded by
	\begin{eqnarray*}
		&&
		|\tilde{\scrT}(\bv,\bu_{1}^{(2)},\ldots,\bu_{1}^{(p)})|\\
		&\footnotemark=&
		\abs*{\sum_{k\neq \pi(1)}\tilde{\lambda}_k
			\langle \tilde{\bu}_k^{(1)},\bv \rangle
			\prod_{q=2}^p\langle \bu_1^{(q)}, \tilde{\bu}_k^{(q)}\rangle}\\
		&\le&\tilde{\lambda}_{\pi(1)}\sum_{k\neq \pi(1)} 
		\prod_{q=2}^p|\langle \bu_1^{(q)}, \tilde{\bu}_k^{(q)}\rangle|\\
		&\le&\tilde{\lambda}_{\pi(1)}
		\cdot
		\prod_{q=2}^p\left(\sum_{k\neq \pi(1)} |\langle \bu_1^{(q)}, \tilde{\bu}_k^{(q)}\rangle|^2\right)^{1/2}\\
		&\le&\lambda_1
		\prod_{q=2}^p\sin\angle (\bu_1^{(q)},\tilde{\bu}_{\pi(1)}^{(q)}),
	\end{eqnarray*}
	\footnotetext{\label{note2} This equality is true even when $d_i$s are not necessarily equal.}
	where the second inequality follows from Cauchy-Schwarz inequality. The last inequality uses $\tilde{\lambda}_{\pi(1)}\le \lambda_1$.This gives
	\begin{align*}\label{eq:secopert}
		&\sin\angle\left(\tilde{\bu}_{\pi(1)}^{(1)},\,\bu_{1}^{(1)}+\dfrac{1}{\lambda_1}(\tilde{\scrT}-\scrT)(\bI,\bu_{1}^{(2)},\ldots,\bu_{1}^{(p)})\right)\\
		=&\sup_{\bv\in\calS^{d-1},\,\bv\perp \tilde{\bu}_1^{(1)}}
		\abs*{\langle \bv,\,\bu_{1}^{(1)}
			+\dfrac{1}{\lambda_1}(\tilde{\scrT}-\scrT)(\bI,\bu_{1}^{(2)},\ldots,\bu_{1}^{(p)})\rangle}\\
		\le & \prod_{q=2}^p\sin\angle (\bu_1^{(q)},\tilde{\bu}_{\pi(1)}^{(q)}).\numberthis
	\end{align*}
	
	\noindent Moreover,
	\begin{align*}
		\sin\angle(\bu_{\pi(1)}^{(1)},\,\bu_1^{(1)})
		=&\sup_{\bv\in\calS^{d-1},\bv\perp\tilde{\bu}_{\pi(1)}^{(1)}}|\langle \bv,\bu_{1}^{(1)}\rangle|\\
		\le & \dfrac{1}{\lambda_1}
		\|(\scrT-\tilde{\scrT})(\bI,\bu_{1}^{(2)},\ldots,\bu_{1}^{(p)})\|
		+\prod_{q=2}^p\sin\angle (\bu_1^{(q)},\tilde{\bu}_{\pi(1)}^{(q)}).
	\end{align*}
	Similarly for each $q=1,\dots,p$ multiplying both sides by $	\lambda_1\sin\angle(\bu_1^{(1)},\,\bu_{\pi(1)}^{(1)})$, we have
	\begin{align*}
		&\lambda_1\sin^2\angle(\bu_1^{(q)},\,\bu_{\pi(1)}^{(q)})\\
		\le& \,\,\lambda_1\prod_{q=1}^p\sin\angle (\bu_1^{(q)},\tilde{\bu}_{\pi(1)}^{(q)})
		+\|\tilde{\scrT}-\scrT\|\cdot\sin\angle (\bu_1^{(q)},\tilde{\bu}_{\pi(1)}^{(q)})\numberthis.
	\end{align*}
	
	%
	In light of \eqref{eq:prebd}, the first term on the rightmost hand side can be bounded by
	
	$$
	\begin{aligned}
		&\lambda_1\prod_{q=1}^p\sin\angle (\bu_1^{(q)},\tilde{\bu}_{\pi(1)}^{(q)})\\
		\le &\lambda_1\left(\prod_{q=1}^p\sin\angle (\bu_1^{(q)},\tilde{\bu}_{\pi(1)}^{(q)})\right)^{{p-2}\over p}\cdot \max_{1\le q\le p}\sin^2\angle (\bu_1^{(q)},\tilde{\bu}_{\pi(1)}^{(q)})\\
		\le& \lambda_1[1-(1-\lambda_1^{-1}\|\tilde{\scrT}-\scrT\|)^{2/(p-2)}]^{p-2\over 2}\max_{1\le q\le p}\sin^2\angle (\bu_1^{(q)},\tilde{\bu}_{\pi(1)}^{(q)}).
	\end{aligned}
	$$
	Thus, rearranging terms in the above expression gives
	$$
	\max_{1\le q\le p}\sin\angle (\bu_1^{(q)},\tilde{\bu}_{\pi(1)}^{(q)})\le \left(1-[1-(1-\lambda_1^{-1}\|\tilde{\scrT}-\scrT\|)^{2/(p-2)}]^{p-2\over 2}\right)^{-1}\cdot{\|\tilde{\scrT}-\scrT\|\over \lambda_1}.
	$$
	Note that the function
	$$
	h_1(x)=\left(1-[1-(1-x)^{2/(p-2)}]^{p-2\over 2}\right)^{-1}
	$$
	is monotonically increasing and continuously differentiable at $0$ with $h_1(0)=1$. We get
	$$
	\max_{1\le q\le p}\sin\angle (\bu_1^{(q)},\tilde{\bu}_{\pi(1)}^{(q)})\le {(1+\eps)\|\tilde{\scrT}-\scrT\|\over \lambda_1},
	$$
	provided that $\|\tilde{\scrT}-\scrT\|\le c_{\eps,p}\lambda_1$ for any positive numerical constant $c_{\eps,p} < h_1^{-1}(1+\eps)$. Plugging this back into \eqref{eq:secopert} analogously for all $q=1,\dots,p$ we have
	$$
	\max_{1\le q\le p}	\sin\angle\left(\tilde{\bu}_{\pi(1)}^{(q)},\,\bu_{1}^{(q)}
	+\dfrac{1}{\lambda_1}(\tilde{\scrT}-\scrT)\times_{k\neq q}\bu_1^{(q)}
	\right)
	\le  \left(\dfrac{(1+\eps)\|\tilde{\scrT}-\scrT\|}{\lambda_1}\right)^{p-1}.
	$$
	\bigskip
	\paragraph{\textbf{Basic case (b): $\tilde{\lambda}_{\pi(1)}>\lambda_1$}} Next consider the case when $\lambda_1<\tilde{\lambda}_{\pi(1)}$. As in the previous case, we can derive that
	\begin{eqnarray*}
		\lambda_1&\le&\max_{1\le k\le d}\left\{\tilde{\lambda}_k\prod_{q=1}^p |\langle \bu_1^{(q)},\tilde{\bu}_k^{(q)}\rangle|^{(p-2)/p}\right\}
		+|\langle\tilde{\scrT}-\scrT,\,\tilde{\bu}_{\pi(1)}^{(1)}\otimes\dots\otimes \tilde{\bu}_{\pi(1)}^{(p)}\rangle|\\
		&=&\tilde{\lambda}_{\pi(1)}\prod_{q=1}^p |\langle \bu_1^{(q)},\tilde{\bu}_{\pi(1)}^{(q)}\rangle|^{(p-2)/p}
		+|\langle\tilde{\scrT}-\scrT,\,\tilde{\bu}_{\pi(1)}^{(1)}\otimes\dots\otimes \tilde{\bu}_{\pi(1)}^{(p)}\rangle|.
	\end{eqnarray*}
	On the other hand,
	$$
	\tilde{\lambda}_{\pi(1)}\le \tilde{\lambda}_1\le \lambda_1+\|\scrT-\tilde{\scrT}\|,
	$$
	where the second inequality follows from triangular inequality. Therefore
	$$
	\tilde{\lambda}_{\pi(1)}\left(1-\prod_{q=1}^p |\langle \bu_1^{(q)},\tilde{\bu}_{\pi(1)}^{(q)}\rangle|^{(p-2)/p}\right)\le 2\|\tilde{\scrT}-\scrT\|
	$$
	leading to
	
	\begin{align*}
		\prod_{q=1}^p\sin^2\angle (\bu_1^{(q)},\tilde{\bu}_{\pi(1)}^{(q)})
		\le& \left(1-\left({1-2\tilde{\lambda}_{\pi(1)}^{-1}\|\tilde{\scrT}-\scrT\|}\right)^{2/(p-2)}\right)^p\\
		\le& \left(1-\left({1-\lambda_1^{-1}\|\tilde{\scrT}-\scrT\|}\over
		{1+\lambda_1^{-1}\|\tilde{\scrT}-\scrT\|} \right)^{2/(p-2)}\right)^p.
	\end{align*}
	Now following an identical argument as in the previous case, we can get
	\begin{eqnarray*}
		&&\lambda_1\max_{1\le q\le p}\sin^2\angle (\bu_1^{(q)},\tilde{\bu}_{\pi(1)}^{(q)})\\
		&\le&(\lambda_1+\|\tilde{\scrT}-\scrT\|)\left[1-\left({1-\lambda_1^{-1}\|\tilde{\scrT}-\scrT\|}\over
		{1+\lambda_1^{-1}\|\tilde{\scrT}-\scrT\|} \right)^{2/(p-2)}\right]^{p-2\over 2}\times\\
		&&\hskip 20pt\times\max_{1\le q\le p}\sin^2\angle (\bu_1^{(q)},\tilde{\bu}_{\pi(1)}^{(q)})+\|\tilde{\scrT}-\scrT\|\cdot\max_{1\le q\le p}\sin\angle (\bu_1^{(q)},\tilde{\bu}_{\pi(1)}^{(q)}),
	\end{eqnarray*}
	leading to
	\begin{eqnarray*}
		&&\max_{1\le q\le p}\sin\angle (\bu_1^{(q)},\tilde{\bu}_{\pi(1)}^{(q)})\\
		&\le& \left(1-\left(1+\dfrac{\|\tilde{\scrT}-\scrT\|}{\lambda_1}\right)\left[1-\left({1-\lambda_1^{-1}\|\tilde{\scrT}-\scrT\|}\over
		{1+\lambda_1^{-1}\|\tilde{\scrT}-\scrT\|} \right)^{2/(p-2)}\right]^{p-2\over 2}\right)^{-1}{\|\tilde{\scrT}-\scrT\|\over \lambda_1}.
	\end{eqnarray*}
	Note that the function
	$$
	h_2(x)=\left[1-(1+x)\left[1-\left({1-x}\over
	{1+x} \right)^{2/(p-2)}\right]^{p-2\over 2}\right]^{-1}
	$$
	is continuously differentiable at $0$ with $h_2(0)=1$, $h_2'(0)>0$. We get
	$$
	\max_{1\le q\le p}\sin\angle (\bu_1^{(q)},\tilde{\bu}_{\pi(1)}^{(q)})\le {(1+\eps)\|\tilde{\scrT}-\scrT\|\over \lambda_1},
	$$
	provided that $\|\tilde{\scrT}-\scrT\|\le c_{\eps,p}\lambda_1$ for any   positive numerical constant $c_{\eps,p}\le h_2^{-1}(1+\eps)$. Finally, following the same steps as in \eqref{eq:secopert}, we have
	\begin{align*}
		&\sin\angle\left(\tilde{\bu}_{\pi(1)}^{(q)},
		\bu_1^{(q)}
		+\dfrac{1}{\lambda_1}(\tilde{\scrT}-\scrT)\times_{k\neq q}\bu_1^{(k)}\right)\\
		\le & \dfrac{\tilde{\lambda}_{\pi(1)}}{\lambda_1}
		\prod_{k\neq q}\sin\angle\left(\bu_1^{(k)},\tilde{\bu}_{\pi(1)}^{(k)}\right)
		\le  \left(1+\dfrac{\|\tilde{\scrT}-\scrT\|}{\lambda_1}\right)
		\left({(1+\eps)\|\tilde{\scrT}-\scrT\|\over \lambda_1}\right)^{p-1}.
	\end{align*}
	\bigskip
	\paragraph{\textbf{Induction}} Next we treat the more general case by induction. 
	%
	To this end, assume that there exists an injective map $\pi:[l]\to [d]$ such that for all $k\le l (<r)$,
	\begin{equation}
		\label{eq:inductionlam}
		|\lambda_k-\tilde{\lambda}_{\pi(k)}|\le \|\tilde{\scrT}-\scrT\|
	\end{equation}
	and
	\begin{equation}
		\label{eq:inductionu}
		\max_{1\le q\le p}\sin\angle(\bu_k^{(q)},\tilde{\bu}_{\pi(k)}^{(q)})\le {(1+\eps)\|\tilde{\scrT}-\scrT\|\over \lambda_k}
	\end{equation}
	we shall now argue they continue to hold for $k=l+1$.
	
	\paragraph{\textbf{Induction (a): $\lambda_{l+1}\ge \max_{k\notin\pi([l])}\tilde{\lambda}_{k}$}} 
	
	Similar to before,
	\begin{eqnarray*}
		\lambda_{l+1}&=&\langle \scrT,\bu_{l+1}^{(1)}\otimes\cdots\otimes \bu_{l+1}^{(p)}\rangle\\
		&\le&\langle \tilde{\scrT},\bu_{l+1}^{(1)}\otimes\cdots\otimes\bu_{l+1}^{(p)}\rangle+\|\tilde{\scrT}-\scrT\|\\
		&\le&\max_{1\le k\le d}\left\{\tilde{\lambda}_k\prod_{q=1}^p |\langle \bu_{l+1}^{(q)},\tilde{\bu}_k^{(q)}\rangle|^{(p-2)/p}\right\}+\|\tilde{\scrT}-\scrT\|.
	\end{eqnarray*}
	We first argue that the index maximizing the rightmost hand side is not from $\pi([l])$. To this end, note that by the induction hypothesis, for any $k\in [l]$,
	\begin{eqnarray*}
		\tilde{\lambda}_{\pi(k)}\prod_{q=1}^p |\langle \bu_{l+1}^{(q)},\tilde{\bu}_{\pi(k)}^{(q)}\rangle|^{(p-2)/p}&\le&(\lambda_k+\|\tilde{\scrT}-\scrT\|)\left(\max_{1\le q\le p}\sin\angle(\bu_k^{(q)},\tilde{\bu}_{\pi(k)}^{(q)})\right)^{p-2}\\
		&\le&(\lambda_k+\|\tilde{\scrT}-\scrT\|)\left((1+\eps)\|\tilde{\scrT}-\scrT\|/\lambda_k\right)\\
		&\le&(1+\eps)(1+c_{\eps,p})\|\tilde{\scrT}-\scrT\|.
	\end{eqnarray*}
	Therefore,
	\begin{eqnarray*}
		&&\max_{1\le k\le l}\left\{\tilde{\lambda}_{\pi(k)}\prod_{q=1}^p |\langle \bu_{l+1}^{(q)},\tilde{\bu}_{\pi(k)}^{(q)}\rangle|^{(p-2)/p}\right\}+\|\tilde{\scrT}-\scrT\|\\
		&\le& [1+(1+\eps)(1+c_{\eps,p})]\|\tilde{\scrT}-\scrT\|\\
		&\le& c_{\eps,p}[1+(1+\eps)(1+c_{\eps,p})]\lambda_{l+1}<\lambda_{l+1},
	\end{eqnarray*}
	by taking $0<c_{\eps, p}<h_3^{-1}(1)$ for the function $h_3(x)=x[1+(1+\eps)(1+x)]$. Thus the index, hereafter denoted by $\pi(l+1)$, that maximizes
	$$
	\tilde{\lambda}_k\prod_{q=1}^p |\langle \bu_{l+1}^{(q)},\tilde{\bu}_k^{(q)}\rangle|^{(p-2)/p}
	$$
	must be different from $\{\pi(1),\ldots,\pi(l)\}$. In addition, because
	$$
	\tilde{\lambda}_{\pi(l+1)}\le \lambda_{l+1}\le \tilde{\lambda}_{\pi(l+1)}\prod_{q=1}^p |\langle \bu_{l+1}^{(q)},\tilde{\bu}_{\pi(l+1)}^{(q)}\rangle|^{(p-2)/p}+\|\tilde{\scrT}-\scrT\|\le \tilde{\lambda}_{\pi(l+1)}+\|\tilde{\scrT}-\scrT\|,
	$$
	we immediately deduce that
	$$
	|\tilde{\lambda}_{\pi(l+1)}-\lambda_{l+1}|\le \|\tilde{\scrT}-\scrT\|,
	$$
	and
	\begin{equation}
		\label{eq:prelimbdu}
		\left(\prod_{q=1}^p|\langle \bu_{l+1}^{(q)},\tilde{\bu}_{\pi(l+1)}^{(q)}\rangle|\right)^{1/p}
		\ge \left(1-\lambda_{l+1}^{-1}\|\tilde{\scrT}-\scrT\|\right)^{1/(p-2)}.
	\end{equation}
	
	\noindent Similar to before, we can derive
	\begin{eqnarray*}
		\lambda_{l+1}\sin^2\angle(\bu_{l+1}^{(1)},\tilde{\bu}_{\pi(l+1)}^{(1)})
		&\,\,\footnoteref{note2}=&
		\scrT(\bu_{l+1}^{(1)}-\langle\bu_{l+1}^{(1)},\tilde{\bu}_{\pi(l+1)}^{(1)}\rangle\tilde{\bu}_{\pi(l+1)}^{(1)},\bu_{l+1}^{(2)},\ldots,\bu_{l+1}^{(p)})\\
		&\le&\tilde{\scrT}(\bu_{l+1}^{(1)}-\langle\bu_{l+1}^{(1)},\tilde{\bu}_{\pi(l+1)}^{(1)}\rangle\tilde{\bu}_{\pi(l+1)}^{(1)},\bu_{l+1}^{(2)},\ldots,\bu_{l+1}^{(p)})\\
		&&\hskip 50pt+\|\tilde{\scrT}-\scrT\|\sin\angle(\bu_{l+1}^{(1)},\tilde{\bu}_{\pi(l+1)}^{(1)}).
	\end{eqnarray*}
	Moreover, because $\lambda_{l+1}\ge \max_{k\notin\pi([l])}\tilde{\lambda}_{k}$, we get
	\begin{eqnarray*}
		&&\tilde{\scrT}(\bu_{l+1}^{(1)}-\langle\bu_{l+1}^{(1)},\tilde{\bu}_{\pi(l+1)}^{(1)}\rangle\tilde{\bu}_{\pi(l+1)}^{(1)},\bu_{l+1}^{(2)},\ldots,\bu_{l+1}^{(p)})\\
		&\footnoteref{note2}=&\sum_{k\neq \pi(l+1)} \tilde{\lambda}_k\prod_{q=1}^p\langle \bu_{l+1}^{(q)}, \tilde{\bu}_k^{(q)}\rangle\\
		&\le&\sum_{k=1}^l \tilde{\lambda}_{\pi(k)}\prod_{q=1}^p\langle \bu_{l+1}^{(q)}, \tilde{\bu}_{\pi(k)}^{(q)}\rangle+\lambda_{l+1}\sum_{k\notin \pi([l+1])} \prod_{q=1}^p|\langle \bu_{l+1}^{(q)}, \tilde{\bu}_k^{(q)}\rangle|\\
		&\le&\sum_{k=1}^l \tilde{\lambda}_{\pi(k)}\prod_{q=1}^p\langle \bu_{l+1}^{(q)}, \tilde{\bu}_{\pi(k)}^{(q)}\rangle
		+\lambda_{l+1}\sin\angle(\bu^{(1)}_{l+1},\tilde{\bu}^{(1)}_{\pi(l+1)})
		\left(\sum_{k\notin \pi([l+1])}\prod_{q=2}^p|\langle \bu_{l+1}^{(q)}, \tilde{\bu}_{\pi(k)}^{(q)} \rangle |^2\right)^{1/2} \\
		&\le&\sum_{k=1}^l \tilde{\lambda}_{\pi(k)}\prod_{q=1}^p\langle \bu_{l+1}^{(q)}, \tilde{\bu}_{\pi(k)}^{(q)}\rangle
		+\lambda_{l+1}\prod_{q=1}^p\sin\angle(\bu^{(q)}_{l+1},\tilde{\bu}^{(q)}_{\pi(l+1)}).
	\end{eqnarray*}
	The first term on the rightmost hand side can be bounded by
	\begin{align*}\label{eq:prevterms}
		&\sum_{k=1}^l \tilde{\lambda}_{\pi(k)}\prod_{q=1}^p\langle \bu_{l+1}^{(q)}, \tilde{\bu}_{\pi(k)}^{(q)}\rangle\\
		\le&\max_{1\le k\le l}\{\tilde{\lambda}_{\pi(k)}\sin\angle (\bu_k^{(1)},\tilde{\bu}_{\pi(k)}^{(1)})\}\left(\max_{1\le q\le p}\sin\angle(\bu_{l+1}^{(q)}, \tilde{\bu}_{\pi(l+1)}^{(q)})\right)^{p-1}\\
		\le&\max_{1\le k\le l}\left\{(\lambda_k+\|\tilde{\scrT}-\scrT\|)\sin\angle (\bu_k^{(1)},\tilde{\bu}_{\pi(k)}^{(1)})\right\}\left(\max_{1\le q\le p}\sin\angle(\bu_{l+1}^{(q)}, \tilde{\bu}_{\pi(l+1)}^{(q)})\right)^{p-1}\\
		\le&(1+\|\tilde{\scrT}-\scrT\|/\lambda_l)(1+\eps)\|\tilde{\scrT}-\scrT\|\left(\max_{1\le q\le p}\sin\angle(\bu_{l+1}^{(q)}, \tilde{\bu}_{\pi(l+1)}^{(q)})\right)^{p-1}\numberthis\\
		\le&(1+\|\tilde{\scrT}-\scrT\|/\lambda_l)(1+\eps)\|\tilde{\scrT}-\scrT\|\left(\max_{1\le q\le p}\sin\angle(\bu_{l+1}^{(q)}, \tilde{\bu}_{\pi(l+1)}^{(q)})\right)^2.
	\end{align*}
	On the other hand, as before, we can derive from \eqref{eq:prelimbdu} that
	$$
	\prod_{q=1}^p\sin^2\angle (\bu_{l+1}^{(q)},\tilde{\bu}_{\pi(l+1)}^{(q)})\le \left(1-\left(1-{\|\tilde{\scrT}-\scrT\|\over \lambda_{l+1}}\right)^{2/(p-2)}\right)^p,
	$$
	so that the second term can be bounded by
	\begin{eqnarray*}
		&&\lambda_{l+1}\prod_{q=1}^p\sin\angle(\bu^{(q)}_{l+1},\tilde{\bu}^{(q)}_{\pi(l+1)})\\
		&\le&\lambda_{l+1}\left[1-\left(1-{\|\tilde{\scrT}-\scrT\|\over \lambda_{l+1}}\right)^{2/(p-2)}\right]^{p-2\over 2}\max_{1\le q\le p}\sin^2\angle(\bu_{l+1}^{(q)}, \tilde{\bu}_{\pi(l+1)}^{(q)}).
	\end{eqnarray*}
	Denote by
	$$
	h_4(x;\eps,p)=(1+x)\left[1-\left({1-x}\over{1+x}\right)^{2/(p-2)}\right]^{p-2\over 2}+(1+\eps)x(1+x).
	$$
	Then since $\lambda_{l+1}\le \lambda_l$
	\begin{eqnarray*}
		&\lambda_{l+1}\sin^2\angle(\bu_{l+1}^{(1)},\tilde{\bu}_{\pi(l+1)}^{(1)})\le &\lambda_{l+1}h_4\left({\|\tilde{\scrT}-\scrT\|\over \lambda_{l+1}}; \eps, p\right)\max_{1\le q\le p}\sin^2\angle (\bu_1^{(q)},\tilde{\bu}_{\pi(1)}^{(q)})\\
		&&+\|\tilde{\scrT}-\scrT\|\cdot\max_{1\le q\le p}\sin\angle (\bu_1^{(q)},\tilde{\bu}_{\pi(1)}^{(q)}),
	\end{eqnarray*}
	implying
	$$
	\sin\angle(\bu_{l+1}^{(1)},\tilde{\bu}_{\pi(l+1)}^{(1)})\le{1\over 1- h_4\left(\|\tilde{\scrT}-\scrT\|/\lambda_{l+1}; \eps, p\right)}\cdot{\|\tilde{\scrT}-\scrT\|\over\lambda_{l+1}}.
	$$
	Observe that $h_4$ is a continuous and increasing function of $x$ and $h_4(0)=0$. Provided that $\|\tilde{\scrT}-\scrT\|\le c_{\eps,p}\lambda_{l+1}$ for some positive numerical constant $c_{\eps,p}\le h_4^{-1}(\eps/(1+\eps))$, we get
	\begin{equation}\label{eq:sineind}
		\left(\max_{1\le q\le p}\sin\angle(\bu_{l+1}^{(q)}, \tilde{\bu}_{\pi(l+1)}^{(q)})\right)\le \dfrac{(1+\eps)\|\tilde{\scrT}-\scrT\|}{\lambda_{l+1}}.
	\end{equation}
	Finally, for any unit vector $\bv\perp \tilde{\bu}_{\pi(l+1)}^{(q)}$ we can derive
	\begin{align*}\label{eq:indsecord}
		\lambda_{l+1}\langle\bv,\bu^{(q)}\rangle=&\scrT\times_q\bv\times_{k\neq q}\bu_{l+1}^{(k)}\\
		=&(\scrT-\tilde{\scrT})\times_q\bv\times_{k\neq q}\bu_{l+1}^{(k)}
		+\tilde{\scrT}\times_q\bv\times_{k\neq q}\bu_{l+1}^{(k)}.\numberthis
	\end{align*}
	Then by similar calculation as above, we can bound
	\begin{align*}
		&|\tilde{\scrT}\times_q\bv\times_{k\neq q}\bu_{l+1}^{(k)}|\\
		\le & \sum_{s=1}^l\tilde{\lambda}_{\pi(s)}\prod_{k\neq q}|\langle \bu_{l+1}^{(k)},\,\tilde{\bu}_{l+1}^{(k)}\rangle|+\lambda_{l+1}\prod_{k\neq q}\sin\angle \left(\bu_{l+1}^{(k)},\,\tilde{\bu}_{\pi(l+1)}^{(k)}\right).
	\end{align*}
	Following \eqref{eq:prevterms}, we have
	\begin{align*}
		&\sum_{s=1}^l\tilde{\lambda}_{\pi(s)}\prod_{k\neq q}|\langle \bu_{l+1}^{(k)},\,\tilde{\bu}_{l+1}^{(k)}\rangle|\\
		\le& (1+\|\tilde{\scrT}-\scrT\|/\lambda_l)(1+\eps)\|\tilde{\scrT}-\scrT\|
		\left(\max_{1\le q\le p}\sin\angle(\bu_{l+1}^{(q)},\,\tilde{\bu}_{\pi(l+1)}^{(q)})\right)^{p-2}\\
		\le& \lambda_{l+1}\left(1+{\|\tilde{\scrT}-\scrT\|\over\lambda_l}\right)
		\left(\dfrac{(1+\eps)\|\tilde{\scrT}-\scrT\|}{\lambda_{l+1}}\right)^{p-1},
	\end{align*}
	where we use \eqref{eq:sineind} in the last line. Now taking supremum over $\bv$ on both sides of \eqref{eq:indsecord}, we have
	\begin{align*}
		&\sin\angle\left(\tilde{\bu}_{\pi(l+1)}^{(q)},\,\bu_{(l+1)}^{(q)}
		+\dfrac{1}{\lambda_{l+1}}(\tilde{\scrT}-\scrT)\times_{k\neq q}
		\bu_{l+1}^{(k)}
		\right)\\
		=&\sup_{\bv\in\calS^{d-1}, \bv\perp\tilde{\bu}_{\pi(l+1)}^{(q)}}
		\abs*{\langle \bv,\,\bu_{(l+1)}^{(1)}
			+\dfrac{1}{\lambda_{l+1}}(\tilde{\scrT}-\scrT)\times_{k\neq q}
			\bu_{l+1}^{(k)}
			\rangle}\\
		\le & \dfrac{1}{\lambda_{l+1}}\sup_{\bv\in\calS^{d-1}, \bv\perp\tilde{\bu}_{\pi(l+1)}^{(q)}}
		{\tilde{\scrT}}\times_q\bv\times_{k\neq q}\bu_{l+1}^{(k)}\\
		\le & \left(2+{\|\tilde{\scrT}-\scrT\|\over\lambda_l}\right)
		\left(\dfrac{(1+\eps)\|\tilde{\scrT}-\scrT\|}{\lambda_{l+1}}\right)^{p-1}.
	\end{align*}
	
	\bigskip
	\paragraph{\textbf{Induction (b): $\lambda_{l+1}< \max_{k\notin\pi([l])}\tilde{\lambda}_{k}$}} Write $\tilde{\bU}_l^{(1)}=(\tilde{\bu}^{(1)}_{\pi(1)},\ldots,\tilde{\bu}^{(1)}_{\pi(l)})$. Then
	\begin{eqnarray*}
		\max_{k\notin\pi([l])}\tilde{\lambda}_{k}&=&\max_{\substack{\ba^{(q)}\in\calS^{d-1}, 1\le q\le p\\ (\tilde{\bU}_l^{(1)})^\top\ba^{(1)}=0}}\langle \tilde{\scrT}, \ba^{(1)}\otimes\cdots\otimes\ba^{(p)}\rangle\\
		&\le&\max_{\substack{\ba^{(q)}\in\calS^{d-1}, 1\le q\le p\\ (\tilde{\bU}_l^{(1)})^\top\ba^{(1)}=0}}\langle \scrT, \ba^{(1)}\otimes\cdots\otimes\ba^{(p)}\rangle+\|\tilde{\scrT}-\scrT\|.
	\end{eqnarray*}
	Observe that
	\begin{eqnarray*}
		&&\max_{\substack{\ba^{(q)}\in\calS^{d-1}, 1\le q\le p\\ (\tilde{\bU}_l^{(1)})^\top\ba^{(1)}=0}}\langle {\scrT}, \ba^{(1)}\otimes\cdots\otimes\ba^{(p)}\rangle\\
		&=&\max_{\ba^{(q)}\in\calS^{d-1}, 1\le q\le p}\langle {\scrT}(I-\tilde{\bU}_l^{(1)}(\tilde{\bU}_l^{(1)})^\top,I,\ldots,I), \ba^{(1)}\otimes\cdots\otimes\ba^{(p)}\rangle\\
		&=&\max_{\ba^{(q)}\in\calS^{d-1}, 1\le q\le p}\left\langle \sum_{k=1}^d{\lambda}_k(I-\tilde{\bU}_l^{(1)}(\tilde{\bU}_l^{(1)})^\top){\bu}_k^{(1)}\otimes\cdots\otimes{\bu}_k^{(p)}, \ba^{(1)}\otimes\cdots\otimes\ba^{(p)}\right\rangle\\
		&=&\max_{1\le k\le d}\{{\lambda}_k\|(I-\tilde{\bU}_l^{(1)}(\tilde{\bU}_l^{(1)})^\top){\bu}_k^{(1)}\|\}.
	\end{eqnarray*}
	By the induction hypothesis, for any $k\le l$,
	\begin{eqnarray*}
		\lambda_k\|(I-\tilde{\bU}_l^{(1)}(\tilde{\bU}_l^{(1)})^\top)\bu_k^{(1)}\|&\le& \lambda_k\|(I-\tilde{\bu}_{\pi(k)}^{(1)}(\tilde{\bu}_{\pi(k)}^{(1)})^\top)\bu_k^{(1)}\|\\
		&\le& (1+\eps)\|\tilde{\scrT}-\scrT\|\\
		&<&\lambda_{l+1}-\|\tilde{\scrT}-\scrT\|,
	\end{eqnarray*}
	by taking $c_{\eps,p}>0$ small enough. Hence
	$$
	\max_{k\notin\pi([l])}\tilde{\lambda}_{k}\le \max_{k>l}\{{\lambda}_k\|(I-\tilde{\bU}_l^{(1)}(\tilde{\bU}_l^{(1)})^\top){\bu}_k^{(1)}\|\}+\|\tilde{\scrT}-\scrT\|\le \lambda_{l+1}+\|\tilde{\scrT}-\scrT\|.
	$$
	This suggests that the index, denoted by $\pi(l+1)$, that maximizes
	$$
	\left\{\tilde{\lambda}_k\prod_{q=1}^p |\langle \bu_{l+1}^{(q)},\tilde{\bu}_k^{(q)}\rangle|^{(p-2)/p}\right\}.
	$$
	is distinct from $\pi([l])$. Moreover, following the same argument as the previous case, we can derive that
	$$
	\tilde{\lambda}_{\pi(l+1)}-\|\tilde{\scrT}-\scrT\|\le \lambda_{l+1}\le \tilde{\lambda}_{\pi(l+1)}\prod_{q=1}^p |\langle \bu_{l+1}^{(q)},\tilde{\bu}_{\pi(l+1)}^{(q)}\rangle|^{(p-2)/p}+\|\tilde{\scrT}-\scrT\|,
	$$
	so that
	$$
	|\tilde{\lambda}_{\pi(l+1)}-\lambda_{l+1}|\le \|\tilde{\scrT}-\scrT\|,
	$$
	and
	$$
	\left(\prod_{q=1}^p|\langle \bu_{l+1}^{(q)},\tilde{\bu}_{\pi(l+1)}^{(q)}\rangle|\right)^{1/p}\ge \left(1-2\lambda_{l+1}^{-1}\|\tilde{\scrT}-\scrT\|\right)^{1/(p-2)}.
	$$
	The rest of the proof is identical to the previous case and works with the same $c_{\eps,p}$ and is therefore omitted for brevity.
	
	\noindent Gathering the conditions used through the proof, we need the constant 
	$$
	c_{\eps,p}\le \min\{(1+\eps)^{-1},\,h_1^{-1}(1+\eps),\,h_2^{-1}(1+\eps),\,h_3^{-1}(1),\,h_4^{-1}(\eps/(1+\eps))\}
	$$
	where
	\begin{align}\label{eq:cpfns}
		h_1(x)&=\left(1-[1-(1-x)^{2/(p-2)}]^{p-2\over 2}\right)^{-1} \nonumber \\
		h_2(x)&=\left[1-(1+x)\left[1-\left({1-x}\over{1+x}\right)^{2/(p-2)}\right]^{p-2\over 2}\right]^{-1} \\
		h_3(x;\eps)&= x[1+(1+\eps)(1+x)] \nonumber \\
		h_4(x;\eps)&=(1+x)\left[1-\left({1-x}\over{1+x}\right)^{2/(p-2)}\right]^{p-2\over 2}+(1+\eps)x(1+x) \nonumber
	\end{align}
	
	Note that although for preciseness, in the proof, we take the constant $c_{\eps, p}>0$ depending on the order of the tensor, it can be taken to be strictly increasing with $p$ so that the argument holds if we take $c_{\eps, 3}$ for all $p\ge 3$.
\end{proof}

\vskip 20pt

\begin{proof}[Proof of Theorem \ref{th:allsingpert}] The proof uses the following Lemma which is proved in the appendix.
	
	\begin{lemma}\label{le:specnorm} Under the assumptions of Corollary \ref{th:allsingpert}, there exist a numerical constant $C>0$, a permutation $\pi:[d_{\min}]\to [d_{\min}]$ and vectors $\boldsymbol{\gamma}^{(q)}\in \{+1,-1\}^{d_{\min}}$ such that the $d_q\times k$ orthogonal matrices 
		$$\bV^{(q)}_k=[\bu_1^{(q)}\,\dots\,\bu_{k}^{(q)}]\,\,\text{ and}\,\,
		\tilde{\bV}^{(q)}_k=[\gamma_1^{(q)}\tilde{\bu}_{\pi(1)}^{(q)}\,\dots\,\gamma_{k}^{(q)}\tilde{\bu}_{\pi(k)}^{(q)}]$$ satisfy
		\begin{equation}\label{eq:specnorm}
			\|\bV^{(q)}_k-\tilde{\bV}^{(q)}_k\|\le {C\|\tilde{\scrT}-\scrT\|\over{\lambda_k}}
		\end{equation}
		for $1\le k\le r$.
	\end{lemma}
	
	\medskip
	
	We first consider the case where $\lambda>0$. Following Lemma \ref{le:specnorm}, real singular vectors of $\scrT$ and $\tilde{\scrT}$ with nonzero singular values $\lambda$ and $\tilde{\lambda}$ can be written as:
	$$
	\bv^{(q)}=\bV^{(q)}_{r}\ba\quad\text{and}\quad\tilde{\bv}^{(q)}=\tilde{\bV}^{(q)}_{r}\tilde{\ba},
	$$
	where $\ba$, $\tilde{\ba}\in \RR^{r}$ are two unit vectors with $|\ba_k|=\left(\tfrac{\lambda}{\lambda_k}\right)^{1/(p-2)}\mathbbm{1}(|\ba_k|\neq 0)$ and  $|\tilde{\ba}_k|=\left(\tfrac{\tilde{\lambda}}{\tilde{\lambda}_{\pi(k)}}\right)^{1/(p-2)}\mathbbm{1}(|\tilde{\ba}_k|\neq 0)$ for $1\le k\le d_{\min}$. Given a singular vector tuple $\bv^{(q)}$, with the active set $S=\{k:|\langle \bu^{(q)},\,\bv^{(q)}\rangle|>0\}$ (note that $\bv^{(q)}$ have the same active set). Let $\{\tilde{\bv}^{(q)}\}$ be a corresponding singular vector tuple with signs
	$$
	{\rm sign}\left(\langle\tilde{\bv}^{(q)},\,\tilde{\bu}^{(q)}_{\pi(k)}\rangle\right)
	=\gamma_k^{(q)}{\rm sign}\left(\langle\bv^{(q)},\,\bu^{(q)}_k\rangle\right)
	$$
	where $\gamma_k^{(q)}$ and $\pi$ are respectively the signs and permutation from Lemma \ref{le:specnorm}. We will now show that $|\lambda-\tilde{\lambda}|$ and $\|\bv^{(q)}-\tilde{\bv}^{(q)}\|$ are small. For notational convenience we prove only for the case where
	$$
	{\rm sign}\left(\langle\tilde{\bv}^{(q)},\,\tilde{\bu}^{(q)}_{\pi(k)}\rangle\right)
	={\rm sign}\left(\langle\bv^{(q)},\,\bu^{(q)}_k\rangle\right)=1
	$$
	for all $1\le q\le p$ and $1\le k \le d_{\min}$. The result for other possible signs will follow similarly. To begin with, we have
	\begin{align*}\label{eq:singtri}
		\|\bv^{(q)}-\tilde{\bv}^{(q)}\|\le & \|\bV^{(q)}_K-\tilde{\bV}^{(q)}_K\|+\|\tilde{\bV}^{(q)}_K\|\|\tilde{\ba}-\ba\|\numberthis
	\end{align*}
	where $K=\max\{k:|\langle \bu^{(q)},\,\bv^{(q)}\rangle|>0\}$. Notice that $\lambda_K=\min\{\lambda_k:|\langle \bu^{(q)},\,\bv^{(q)}\rangle|>0\}$. By Lemma \ref{le:specnorm}, 
	\begin{equation}\label{eq:matpart}
		\|\bV^{(q)}_K-\tilde{\bV}^{(q)}_K\|\le \dfrac{C\|\tilde{\scrT}-\scrT\|}{\lambda_K}=\dfrac{C\|\tilde{\scrT}-\scrT\|}{\lambda_{\min}^*}.
	\end{equation}
	In light of \eqref{eq:singtri}, it is thus enough to show the upper bound on $\|\tilde{\ba}-\ba\|$. Writing $x=\ba_k=\lambda_k^{1/(p-2)}$ and $y=\tilde{\ba}_k=\tilde{\lambda}_k^{1/(p-2)}$, we have
	\begin{align*}
		\abs*{\dfrac{1}{\lambda_k^{1/(p-2)}}-\dfrac{1}{\lambda_{\pi(k)}^{1/(p-2)}}}
		=&\dfrac{|x-y|}{xy}\\
		=&\dfrac{|x^{p-2}-y^{p-2}|}{xy(x^{p-3}+x^{p-4}y+\dots+xy^{p-4}+y^{p-3})}\\
		\le & \dfrac{|x^{p-2}-y^{p-2}|}{x^{p-2}y+y^{p-2}x}\\
		\le & \min\left\{\dfrac{|\lambda_k-\tilde{\lambda}_{\pi(k)}|}{\lambda_k(\tilde{\lambda}_{\pi(k)})^{1/(p-2)}},\dfrac{|\lambda_k-\tilde{\lambda}_{\pi(k)}|}
		{\tilde{\lambda}_{\pi(k)}(\lambda_{k})^{1/(p-2)}}\right\}\\
		\le &\dfrac{C\|\tilde{\scrT}-\scrT\|}{\lambda_k}
		\dfrac{1}{(\tilde{\lambda}_{\pi(k)})^{1/(p-2)}},
	\end{align*}
	by Theorem \ref{th:ortho-perturb}, and thus for $S=\{k:\ba_k\neq 0\}$,
	\begin{align*}
		\sum_{k\in S}\abs*{\dfrac{1}{\lambda_k^{1/(p-2)}}-\dfrac{1}{\lambda_{\pi(k)}^{1/(p-2)}}}^2
		\le& \dfrac{C\|\tilde{\scrT}-\scrT\|^2}{\left(\lambda_{\min}^*\right)^2}
		\cdot\sum_{k\in S}\dfrac{1}{(\tilde{\lambda}_{\pi(k)})^{2/(p-2)}}\\
		=& \dfrac{C\|\tilde{\scrT}-\scrT\|^2}{\left(\lambda_{\min}^*\right)^2}
		\cdot\dfrac{1}{(\tilde{\lambda})^{2/(p-2)}}.
	\end{align*}
	We have showed that
	\begin{equation}\label{eq:scalediff}
		\norm*{\dfrac{1}{\lambda^{1/(p-2)}}\ba-\dfrac{1}{\tilde{\lambda}^{1/(p-2)}}\tilde{\ba}}
		\le \dfrac{C\|\tilde{\scrT}-\scrT\|}{\lambda_{\min}^*}\cdot \dfrac{1}{\tilde{\lambda}^{1/(p-2)}}.
	\end{equation}
	Let us assume without loss of generality that $\tilde{\lambda}<\lambda$. By an analogous calculation, we can also show that
	\begin{equation}\label{eq:lamb_diff}
		\abs*{\dfrac{1}{{\lambda}^{1/(p-2)}}-\dfrac{1}{\tilde{\lambda}^{1/(p-2)}}}\le \dfrac{C\|\tilde{\scrT}-\scrT\|}{\lambda_{\min}^*}\cdot \dfrac{1}{\tilde{\lambda}^{1/(p-2)}},
	\end{equation}
	which when combined with \eqref{eq:scalediff} and the fact that $\|\ba\|=\|\tilde{\ba}\|=1$ implies
	\begin{align*}
		\|\ba-\tilde{\ba}\|
		\le&\, \tilde{\lambda}^{1/(p-2)}\abs*{\dfrac{1}{{\lambda}^{1/(p-2)}}-\dfrac{1}{\tilde{\lambda}^{1/(p-2)}}}\|\ba\|+\tilde{\lambda}^{1/(p-2)}
		\norm*{\dfrac{1}{\lambda^{1/(p-2)}}\ba-\dfrac{1}{\tilde{\lambda}^{1/(p-2)}}\tilde{\ba}}\\
		\le& \, \dfrac{C\|\tilde{\scrT}-\scrT\|}{\lambda_{\min}^*}.
	\end{align*}
	Plugging this bound back into \eqref{eq:singtri} along with \eqref{eq:matpart} finishes the proof for singular vectors. 
	
	For the singular values $\lambda$ and $\tilde{\lambda}$, note that
	\begin{align*}
		|\lambda-\tilde{\lambda}|
		=&\left|\lambda^{\tfrac{1}{p-2}}-\tilde{\lambda}^{\tfrac{1}{p-2}}\right|
		\left(\lambda^{\tfrac{p-3}{p-2}}+\lambda^{\tfrac{p-4}{p-2}}
		\tilde{\lambda}^{\tfrac{1}{p-2}}
		+\dots+
		\tilde{\lambda}^{\tfrac{p-3}{p-2}}
		\right)\\
		\le & p\lambda^{\tfrac{p-3}{p-2}}\cdot\dfrac{C\|\tilde{\scrT}-\scrT\|}{\lambda^*_{\min}}\cdot\lambda^{\tfrac{1}{p-2}}\\
		= & p\lambda  \cdot\dfrac{C\|\tilde{\scrT}-\scrT\|}{\lambda^*_{\min}}.
	\end{align*}
	where we use that $\lambda>\tilde{\lambda}$ and \eqref{eq:lamb_diff} in the third inequality.
	Note that 
	\begin{align*}
		\dfrac{1}{\lambda^{1/(p-2)}}=\sum_{k\in S}\dfrac{1}{\lambda_k^{1/(p-2)}}>\dfrac{1}{\left(\lambda_{\min}^*\right)^{1/(p-2)}}
	\end{align*}
	implying that $\lambda<\lambda_{\min}^*$. Thus $|\lambda-\tilde{\lambda}|\le C\|\tilde{\scrT}-\scrT\|$.
	
	Now consider the case when $\lambda=0$. Any set of unit vectors $\bw^{(1)},\dots,\bw^{(p)}$ is a singular vector tuple of 
	$$
	\sum_{k=1}^{d_{\min}}\lambda_k\be_k^{(1)}\otimes \dots\otimes \be_k^{(p)}.
	$$ 
	corresponding to singular value $\lambda=0$ if and only if $\langle\bw^{(q)},\be_k^{(q)}\rangle =0$ for at least two values of $q\in\{1,\dots ,p\}$. Such singular vectors of $\scrT$ and $\tilde{\scrT}$ can be written as $\bv^{(q)}=\bV_{d_{\min}}^{(q)}\bw^{(q)}$ and $\tilde{\bv}^{(q)}=\tilde{\bV}_{d_{\min}}^{(q)}\bw^{(q)}$. The conclusion then follows directly from Lemma \ref{le:specnorm}. If $\langle \bw^{(q)},\be_k^{(q)} \rangle$ for some $k$ such that $\lambda_k=0$, we use the vacuous bound
	$$
	\|\bv^{(q)}-\tilde{\bv}^{(q)}\|=1\le \dfrac{C\|\tilde{\scrT}-\scrT\|}{\lambda_k}
	$$
	with the convention $1/0=+\infty$.
\end{proof}

\bigskip

\begin{proof}[Proof of  Theorem \ref{th:incoherent}] We show the results using $\scrX$ and its odeco approximation $\scrT$. The analogous results for $\tilde{\scrX}$ and $\tilde{\scrT}$ follow similarly. Recall that the polar factor of $\bA^{(q)}$ is the unitary matrix
	$$\bU^{(q)}=\bA^{(q)}[(\bA^{(q)})^\top \bA^{(q)}]^{-1/2}.$$
	It is not hard to see that
	\begin{equation}
		\label{eq:polarnorm}
		\|\bA^{(q)}-\bU^{(q)}\|= \|[(\bA^{(q)})^\top \bA^{(q)}]^{1/2}-I\|\le \max_{1\le i \le d}|\lambda_i((\bA^{(q)})^\top \bA^{(q)})^{1/2}-1|\le\delta.	
	\end{equation}
	We can then consider approximating $\scrX$by
	$$
	\scrT=\sum_{i=1}^d \eta_i \bu_i^{(1)}\otimes\cdots\bu_i^{(p)}.
	$$
	Recall that
	$$
	\|\scrX-\scrT\|=\sup_{\bx^{(q)}\in \calS^{d-1}: 1\le q\le p}\langle \scrX-\scrT,\bx^{(1)}\otimes\cdots\otimes\bx^{(p)}\rangle.
	$$
	For any fixed $\bx^{(q)}$s,
	\begin{eqnarray*}
		&&\langle \scrX-\scrT,\bx^{(1)}\otimes\cdots\otimes\bx^{(p)}\rangle\\
		&=&\sum_{i=1}^d \eta_i\left(\prod_{q=1}^p\langle \bx^{(q)},\ba_i^{(q)}\rangle-\prod_{q=1}^p\langle \bx^{(q)},\bu_i^{(q)}\rangle\right)\\
		&=&\sum_{i=1}^d \eta_i\left(\langle \bx^{(q)},\ba_i^{(q)}\rangle-\langle \bx^{(1)},\bu_i^{(1)}\rangle\right)\prod_{q=2}^p\langle \bx^{(q)},\ba_i^{(q)}\rangle\\
		&&\hskip 20pt +\sum_{i=1}^d \eta_i\langle \bx^{(1)},\bu_i^{(1)}\rangle\left(\langle \bx^{(2)},\ba_i^{(2)}\rangle-\langle \bx^{(2)},\bu^{(2)}\rangle\right)\prod_{q=3}^p\langle \bx^{(q)},\ba_i^{(q)}\rangle\\
		&&\hskip 20pt +\ldots\ldots+\\
		&&\hskip 20pt +\sum_{i=1}^d \eta_i\prod_{q=1}^{p-1}\langle \bx^{(q)},\bu_i^{(q)}\rangle\left(\langle \bx^{(p)},\ba_i^{(p)}\rangle-\langle \bx^{(p)},\bu_i^{(p)}\rangle\right).
	\end{eqnarray*}
	Each term on the rightmost hand side can be bounded via Cauchy-Schwarz inequality:
	\begin{eqnarray*}
		&&\sum_{i=1}^d \eta_i\prod_{q=1}^{k-1}\langle \bx^{(q)},\bu_i^{(q)}\rangle\left(\langle \bx^{(k)},\ba_i^{(k)}\rangle-\langle \bx^{(k)},\bu_i^{(k)}\rangle\right)\prod_{q=k+1}^p\langle \bx^{(q)},\ba_i^{(q)}\rangle\\
		&\le&\|\bA^{(k)}-\bU^{(k)}\|\left[\sum_{i=1}^d \left(\eta_i^2\prod_{q=1}^{k-1}\langle \bx^{(q)},\bu_i^{(q)}\rangle^2\prod_{q=k+1}^p\langle \bx^{(q)},\ba_i^{(q)}\rangle^2\right)\right]^{1/2}\\
		&\le&\eta_1\|\bA^{(k)}-\bU^{(k)}\|\left[\sum_{i=1}^d \left(\prod_{q=1}^{k-1}\langle \bx^{(q)},\bu_i^{(q)}\rangle^2\prod_{q=k+1}^p\langle \bx^{(q)},\ba_i^{(q)}\rangle^2\right)\right]^{1/2}\\
		&\le&\delta\eta \left[\sum_{i=1}^d \left(\prod_{q=1}^{k-1}\langle \bx^{(q)},\bu_i^{(q)}\rangle^2\prod_{q=k+1}^p\langle \bx^{(q)},\ba_i^{(q)}\rangle^2\right)\right]^{1/2}.
	\end{eqnarray*}
	Note that
	$$
	|\langle \bx^{(q)},\bu_i^{(q)}\rangle|, \qquad |\langle \bx^{(q)},\ba_i^{(q)}\rangle|\le 1,
	$$
	and
	$$
	\sum_{i=1}^d \langle \bx^{(q)},\bu_i^{(q)}\rangle^2=1,\qquad \sum_{i=1}^d \langle \bx^{(q)},\ba_i^{(q)}\rangle^2\le 1+\delta.
	$$
	We immediately get
	$$
	\sum_{i=1}^d \eta_i\left(\langle \bx^{(q)},\ba_i^{(q)}\rangle-\langle \bx^{(1)},\bu_i^{(1)}\rangle\right)\prod_{q=2}^p\langle \bx^{(q)},\ba_i^{(q)}\rangle\le \delta(1+\delta)\eta_1,
	$$
	and for $k\ge 2$,
	$$
	\sum_{i=1}^d \eta_i\prod_{q=1}^{k-1}\langle \bx^{(q)},\bu_i^{(q)}\rangle\left(\langle \bx^{(k)},\ba_i^{(k)}\rangle-\langle \bx^{(k)},\bu_i^{(k)}\rangle\right)\prod_{q=k+1}^p\langle \bx^{(q)},\,\ba_i^{(q)}\rangle \le\delta\eta_1.
	$$
	Hence
	$$
	\|\scrX-\scrT\|\le (p+1)\delta\eta_1.
	$$
	Note also that for any $1\le q \le p,$ using equation \eqref{eq:polarnorm}
	$$\begin{aligned}
		\max_{1\le j \le d}\sin\angle (\ba_j^{(q)},\bu_{\pi(j)}^{(q)})&\le \sqrt{1-\min_{j}\langle \ba_j^{(q)},\,\bu_j^{(q)}\rangle^2 }
		\\&\le\sqrt{1-(1-\|\bA^{(q)}-\bU^{(q)}\|^2/2)}\le \delta/\sqrt{2}.
	\end{aligned}
	$$
	The desired result then follows from  Theorem \ref{th:ortho-perturb}. As mentioned before, the proof for $\tilde{\scrX}$ and $\tilde{\scrT}$ follows by identical steps.
\end{proof}
\vskip 20pt

\begin{proof}[Proof of Corollary \ref{co:incoherent}]
	It is clear from Theorem \ref{th:incoherent} that there exist odeco approximations $\scrT$ and $\tilde{\scrT}$ of $\scrX$ and $\tilde{\scrX}$ respectively, such that
	$$\begin{aligned}
		\|\scrT-\tilde{\scrT}\|\le& \|\scrT-\scrX\|+\|\scrX-\tilde{\scrX}\|+\|\tilde{\scrX}-\tilde{\scrT}\| \\
		\le & (p+1)\delta(\eta_1+\tilde{\eta}_1)+\|\scrX-\tilde{\scrX}\|.
	\end{aligned}$$
	By Theorem \ref{th:odeco-weyl}, there is a permutation $\pi:[d_{\min}]\to[d_{\min}]$ and a constant $C>0$ such that
	$$|\eta_k-\tilde{\eta}_{\pi(k)}|\le C((p+1)\delta(\eta_1+\tilde{\eta}_1)+\|\scrX-\tilde{\scrX}\|)$$
	and
	$$
	\max_{1\le q\le p}\sin\angle\left(\bu_k^{(q)},\tilde{\bu}_{\pi(k)}^{(q)}\right)
	\le C((p+1)\delta(\eta_1+\tilde{\eta}_1)+\|\scrX-\tilde{\scrX}\|+\delta)/\eta_k.
	$$
	Finally, we use the second part of Theorem \ref{th:incoherent} to derive that, by triangle inequality, with the same permutation $\pi$ we also have
	$$
	\max_{1\le q\le p}\sin\angle\left(\ba_k^{(q)},\tilde{\ba}_{\pi(k)}^{(q)}\right)
	\le C((p+1)\delta(\eta_1+\tilde{\eta}_1)+\|\scrX-\tilde{\scrX}\|+\delta)/\eta_k.
	$$	
\end{proof}

\vskip 30pt

\begin{proof}[Proof of Theorem \ref{th:tpcalower}] 
	%
	First note that a lower bound for a special case is also a lower bound for the more general case. Therefore,
	%
	\begin{eqnarray*}
		&&\inf_{\tilde{\bu}_k^{(1)},\ldots,\tilde{\bu}_k^{(p)}}\sup_{\bu_k^{(q)}\in \calS^{d_q-1}: 1\le q\le p}\EE\max_{1\le q\le p} \sin\angle(\bu_k^{(q)},\tilde{\bu}_k^{(q)})\\
		&\ge& \inf_{\tilde{\bu}_k^{(1)},\ldots,\tilde{\bu}_k^{(p)}}\sup_{\substack{\bu_k^{(q)}\in \calS^{d_q-1}: 1\le q\le p\\ \lambda_{k'}=0, \forall k'\neq k}}\EE\max_{1\le q\le p} \sin\angle(\bu_k^{(q)},\tilde{\bu}_k^{(q)}).
	\end{eqnarray*}
	The special case was simply the rank one case where $\scrT$ has only one nonzero singular value $\lambda_k$. It was shown by \cite{zhang2018tensor} that for this case,
	$$
	\inf_{\tilde{\bu}_k^{(1)},\ldots,\tilde{\bu}_k^{(p)}}\sup_{\substack{\bu_k^{(q)}\in \calS^{d_q-1}: 1\le q\le p\\ \lambda_{k'}=0, \forall k'\neq k}}\EE\max_{1\le q\le p} \sin\angle(\bu_k^{(q)},\tilde{\bu}_k^{(q)})\ge c\cdot{\sqrt{d_1+\cdots+d_p}\over \lambda_k},
	$$
	and thus \eqref{eq:svdbd4} follows. The lower bound \eqref{eq:svdbd3} for estimating the singular value follows by the same argument.
\end{proof}

\bibliographystyle{plainnat}
\bibliography{references}

\appendix
\section*{Appendix -- Proof of Theorem \ref{th:comp}}
We first show that if $\lambda_k>0$, then $(\pm\bu_k^{(1)},\ldots,\pm\bu_k^{(p)})$ is a local maximum of $F$. Consider the Lagrange form of $F$:
$$
F_{\lambda_k}(\ba^{(1)},\ldots, \ba^{(p)}):=F(\ba^{(1)},\ldots, \ba^{(p)})+\lambda_k\sum_{q=1}^p\left(1-\|\ba^{(q)}\|^2\right).
$$
It is easy to see that $(\bu_k^{(1)},\ldots,\bu_k^{(p)})$ satisfies the first order condition of $F_{\lambda_k}$:
$$
\scrT\times_{q'\neq q} \ba^{(q')}=\lambda_k \ba^{(q)},\qquad \forall q=1,\ldots, p.
$$
Moreover, it can also be derived that the Hessian of $F_{\lambda_k}$ is
$$
\left[\begin{array}{cccc}-\lambda_k I_{d_1} & \scrT\times_{q'\notin \{1,2\}} \ba^{(q')}&\cdots & \scrT\times_{q'\notin \{1,p\}} \ba^{(q')}\\ (\scrT\times_{q'\notin \{1,2\}} \ba^{(q')})^\top& -\lambda_k I_{d_2}&\ldots& \scrT\times_{q'\notin \{2,p\}} \ba^{(q')}\\ \ldots&\ldots&\ldots&\ldots\\ (\scrT\times_{q'\notin \{1,p\}} \ba^{(q')})^\top& (\scrT\times_{q'\notin \{2,p\}} \ba^{(q')})^\top&\cdots&-\lambda_kI_{d_p} \end{array}\right].
$$
When evaluated at $(\bu_k^{(1)},\ldots,\bu_k^{(p)})$, the Hessian becomes
$$
H=\lambda_k \bv_k\otimes\bv_k -\lambda_k{\rm diag}(I_{d_1}+\bu_k^{(1)}\otimes \bu_k^{(1)},\ldots,I_{d_p}+\bu_k^{(p)}\otimes \bu_k^{(p)}),
$$
where $\bv_k=[(\bu_k^{(1)})^\top,\ldots,(\bu_k^{(p)})^\top]^\top$. For any $(\ba^{(1)},\ldots, \ba^{(p)})\neq (\bu_k^{(1)},\ldots,\bu_k^{(p)})$, write
$$
\ba^{(q)}=\langle \ba^{(q)},\bu_k^{(q)}\rangle\bu_k^{(q)}+ \tilde{\ba}^{(q)}.
$$
It can be verified that
$$
[(\ba^{(1)})^\top,\ldots,(\ba^{(p)})^\top] H [(\ba^{(1)})^\top,\ldots,(\ba^{(p)})^\top]^\top<0
$$
if
$$
\prod_{1\le q\le p}|\langle \ba^{(q)},\bu_k^{(q)}\rangle|<1.
$$
This implies that $(\bu_k^{(1)},\ldots,\bu_k^{(p)})$ is a local maximum of $F$.

We now argue that $F$ has no local maximum other than $\{(\pm\bu_k^{(1)},\ldots,\pm\bu_k^{(p)}): \lambda_k>0\}$.  We shall prove this by contradiction. Assume the contrary that unit length vectors $\ba^{(q)}$s are a local maximum but $(\ba^{(1)},\ldots,\ba^{p})\notin \{(\pm\bu_k^{(1)},\ldots,\pm\bu_k^{(p)}): \lambda_k>0\}$. By first order condition, there exists a $\lambda\in \RR$ such that
$$
\scrT\times_{q'\neq q} \ba^{(q')}=\lambda \ba^{(q)},\qquad \forall q=1,\ldots, p.
$$
We can assume that $\lambda>0$ without loss of generality. We get that for each $k:\lambda_k>0$,
\begin{equation}
	\label{eq:values}
	\abs*{\langle \bu^{(q)}_k,\,\ba^{(q)}\rangle}=\gamma_k=\begin{cases}\left(\dfrac{\lambda}{\lambda_k}\right)^{\frac{1}{p-2}},\forall q=1,\dots,p\text{ if }\prod_{q=1}^p\langle\bu^{(q)}_k,\,\ba_q \rangle\neq 0
		\\0,\hspace{1.5cm}\forall q=1,\dots,p\text{ if }\prod_{q=1}^p\langle\bu^{(q)}_k,\,\ba_q \rangle= 0.
	\end{cases}\tag{A1}
\end{equation}
Moreover, we have
\begin{equation}\label{eq:signs}
	\prod_{q}\langle \bu^{(q)}_k,\ba^{(q)} \rangle\ge 0,\qquad \forall k\text{ such that }\lambda_k> 0.\tag{A2}
\end{equation}

We first consider the case where
$$S:=\{k:\lambda_k\prod_q\langle \bu^{(q)}_k,\ba^{(q)}\rangle\neq 0\}$$
has at most 1 element. We pick $j\in S$ if it exists, otherwise let $j$ be an arbitrary element from $\{k:\lambda_k>0\}.$ Since $\ba^{(q)}\neq \bu_j^{(q)}$ for at least one $q,$ we can construct  a new vector $\bb^{(q)}\in\calS^{d-1},$ which has 
$$\abs*{\langle \bu^{(q)}_j,\,\bb^{(q)}\rangle }> \abs*{\langle \bu^{(q)}_j,\,\ba^{(q)}\rangle },$$ while 
$$\prod_q\abs*{\langle \bu^{(q)}_k,\,\bb^{(q)}\rangle }=0\qquad  \forall k\neq j \rm{\ such\  that\ } \lambda_k>0.$$
It is now easy to see that $F(\bb^{(1)},\dots,\bb^{(p)})> F(\ba^{(1)},\dots,\ba^{(p)}).$ Since we can take $\bb^{(q)}$ arbitrarily close to $\ba^{(q)},$ it is clear that $(\ba^{(1)},\dots,\ba^{(p)})$ cannot be a maximum.

Henceforth, we assume that $S$ has at least two elements, say $j_1$ and $j_2$. Let us define $$\eta:=\min\left\{\abs*{\langle \bu_{j_i}^{(q)}\ba^{(q)}\rangle}:1\le q \le p,\,i=1,2\right\}/2.$$ 
For $0<\delta<\eta,$ for each $1\le q\le p,$ we construct $\bb^{(q)}$ as follows:
\begin{multline*}\bb^{(q)}(\delta)=s_1\left(\sqrt{\langle \bu_{j_1}^{(q)},\,\ba^{(q)}\rangle^2+\delta}\right)\bu^{(q)}_{j_1}
	\\+s_2\left(\sqrt{\langle \bu_{j_2}^{(q)},\,\ba^{(q)}\rangle^2-\delta}\right)\bu^{(q)}_{j_2}+\displaystyle\sum_{k\neq j_1,j_2}\langle \bu^{(q)}_k,\,\ba^{(q)} \rangle\bu^{(q)}_k,
\end{multline*}
where $s_i=\text{sign}(\langle \bu_{j_i}^{(q)},\,\ba^{(q)}\rangle)$ for $i=1,2.$ Evidently, $\bb^{(q)}(\delta)\in\calS^{d_q-1},$ and \\$\text{sign}(\langle \bu^{(q)}_k,\,\bb^{(q)}\rangle)=\text{sign}(\langle \bu^{(q)}_k,\,\ba^{(q)}\rangle)$ for all $k$ and $q.$ Since $(\ba^{(1)},\dots,\ba^{(p)})$ is a critical point, we get using \eqref{eq:values} and  \eqref{eq:signs} that
$$\begin{aligned}
	&F(\bb^{(1)},\dots,\bb^{(p)})-F(\ba^{(1)},\dots,\ba^{(p)})
	\\=&\lambda_{j_1}\prod_{q=1}^p[(\langle \bu^{(q)}_{j_1},\, \ba^{(q)}\rangle)^2+\delta]^{1/2}-\lambda_{j_1}\prod_{q=1}^p\abs*{\langle\bu^{(q)}_{j_1},\,\ba^{(q)}\rangle}+\lambda_{j_2}\prod_{q=1}^p[(\langle \bu^{(q)}_{j_2},\, \ba^{(q)}\rangle)^2-\delta]^{1/2}\\
	&\hskip 50pt -\lambda_{j_2}\prod_{q=1}^p\abs*{\langle\bu^{(q)}_{j_2},\,\ba^{(q)}\rangle}
	\\=&\lambda_{j_1}[\gamma_{j_1}^2+\delta]^{p/2}-\lambda_{j_1}\gamma_{j_1}^p+\lambda_{j_2}[\gamma_{j_2}^2-\delta]^{p/2}-\lambda_{j_2}\gamma_{j_2}^{p}
	\\=&\dfrac{p}{2}\times \lambda_{j_1}\delta \xi_1^{p/2-1}-\dfrac{p}{2}\times \lambda_{j_2}\delta \xi_2^{p/2-1}
\end{aligned}$$
for some $\gamma_{j_1}^2\le \xi_1\le \gamma_{j_1}^2+\delta\text{ and }\gamma_{j_1}^2-\delta\le \xi_2\le \gamma_{j_2}^2.$ Since $f(x)=x^{p/2-1}$ is monotonically increasing for $p>2,$
$$\begin{aligned}
	&F(\bb^{(1)},\dots,\bb^{(p)})-F(\ba^{(1)},\dots,\ba^{(p)})
	\\>&\dfrac{p}{2}\delta[\lambda_{j_1}\gamma_{j_1}^{p-2}-\lambda_{j_2}\gamma_{j_2}^{p-2}]
	\\=&\dfrac{p\delta}{2}\left[\lambda_{j_1}\dfrac{\prod_{q\neq 1}\abs*{\langle \bu^{(q)}_{j_1},\,\ba^{(q)}\rangle}}{\abs*{\langle \bu^{(1)}_{j_1},\,\ba^{(1)}\rangle}}-\lambda_{j_2}\dfrac{\prod_{q\neq 1}\abs*{\langle \bu^{(q)}_{j_2},\,\ba^{(q)}\rangle}}{\abs*{\langle \bu^{(1)}_{j_2},\,\ba^{(1)}\rangle}}\right]
	\\=&\dfrac{p\delta}{2}[\lambda-\lambda]\qquad\text{using \eqref{eq:values}.}
\end{aligned}$$
Since we can take $\delta$ to be arbitrarily close to zero, it is clear that $(\ba^{(1)},\dots,\ba^{(p)})$ is not a local maximum.
\vspace{10pt}

\noindent\textbf{Global attraction of the hidden basis:} We will follow the outline in section 4.2.2 of \cite{belkin2018eigenvectors}. For brevity, we assume $d_1=\dots=d_p=d.$ For $(\ba_1,\dots,\ba_p)\in\calS^{d-1}\times \dots\times \calS^{d-1},$ the tangent space of the cross-product of $p$ spheres is
$$T_{\ba_1,\dots,\ba_p}\calS^{d-1}\times \dots\times \calS^{d-1}=\ba_1^{\perp}\times \dots\times\ba_p^{\perp}.$$
We define the exponential map $\phi:T_{\ba_1,\dots,\ba_p}\calS^{(d-1)\otimes p}\to\calS^{(d-1)\otimes p}$ as:
$$\phi(\bx_1,\dots,\bx_p)=\left(\ba_1\cos\norm*{\bx_1}+\frac{\bx_1}{\norm*{\bx_1}}\sin\norm*{\bx_1},\dots,\ba_p\cos\norm*{\bx_p}+\frac{\bx_p}{\norm*{\bx_p}}\sin\norm*{\bx_p}\right).$$
It can be checked that $D\phi=D\phi^{-1}=\rm{diag}[\calP_{\ba_1^{\perp}},\dots,\calP_{\ba_p^{\perp}}].$ 

We first determine the local convergence manifold of $(\ba_1,\dots,\ba_p),$ that is, the set 
$$\calL_{local}=\{\Tilde{\bx}(0):\lim_{t\to \infty}\bx_i(t)=\ba_i,\,\bx_i(t)\in U_i\,\,\, \forall t\in \mathbb{N}\}$$
for some local neighborhoods $U_i$ of $\ba_i.$ To disprove global attraction to a particular critical point, note that it is enough to determine $\calL_{local}\cap \calQ_{\ba_1}\times \dots\times \calQ_{\ba_p},$ where $$\calQ_{\ba_q}=\{\bv\in\calS^{d-1}:\rm{sign}(\langle \bv,\bu^{(q)}_i\rangle)=\rm{sign}(\langle \ba_q,\bu^{(q)}_i\rangle)\,\,\forall i\text{ such that }\langle \ba_q,\bu^{(q)}_i\rangle\neq 0\}$$ for $1\le q\le p.$

Let $S=\{i:\prod_q\langle \ba_q,\,\bu^{(q)}_i\rangle \neq 0\}.$ We will use $\calP^{(q)}_S=\displaystyle\sum_{i\in S}\bu^{(q)}_i\bu^{(q)T}_i,$ and similarly $\calP^{(q)}_{\bar{S}}$ for $\bar{S}=[d]/S.$ Using \eqref{eq:values} it is easy to see that if $\lambda>0,$ $S$ is in fact same as $S_k=\{i:\langle \ba_k,\,\bu^{(k)}_i\rangle \neq 0\}$ for all $1\le k\le p.$ We then have the following lemma.
\bigskip
\begin{lemma}\label{le:hyperbolic} $D[\phi\circ G\circ \phi^{-1}]_{\phi(\ba_1\dots\ba_p)}$ is a matrix with the following properties:
	\\\noindent1. $D[\phi\circ G\circ \phi^{-1}]$ is the 0 map on $\calK=\{(\bx_1\dots\bx_p):\bx_q\in \rm{Range}(\calP^{(q)}_{\bar{S}})\text{ for all }q\}$.
	\\\noindent2. If $\abs*{S}>1,$ there is a space $$\calL=(\rm{Range}(\calP^{(1)}_{S}\cap \calP_{\ba_1^{\perp}})\cap\calQ_{\ba_1})\times \dots\times (\rm{Range}(\calP^{(p)}_{S}\cap \calP_{\ba_p^{\perp}})\cap\calQ_{\ba_p})$$
	of  positive dimension on which $(\bx_1^T\dots\bx_p^T)[D[\phi\circ G\circ \phi^{-1}]-I](\bx_1^T\dots\bx_p^T)^T>0.$ 
\end{lemma}

\begin{proof}[Proof of Lemma \ref{le:hyperbolic}] Since $(\ba_1,\dots,\ba_p)$ is a fixed point of $G,$ we have using chain rule that 
	$$\begin{aligned}D[\phi\circ G\circ \phi^{-1}]_{\phi(\ba_1\dots\ba_p)}&=D\phi_{G(\ba_1^p) }DG_{\ba_1^p}D\phi^{-1}_{\phi(\ba^p_1)}
		\\&=\text{diag}[\calP_{\ba^{\perp}_1}\dots \calP_{\ba^{\perp}_p}]\,\,DG_{\ba_1^p}\,\,\text{diag}[\calP_{\ba^{\perp}_1}\dots \calP_{\ba^{\perp}_p}].\end{aligned}$$
	Since $\scrT\times_{q'\neq q}\ba^{(q')}=\lambda\ba^{(q)}$ for $q=1,\dots,p,$ $\norm*{\scrT\times _{q'\neq  q}\ba^{(q)}}=\lambda.$ Hence after some calculation we obtain that $DG$ can be written as
	$$\dfrac{1}{\lambda}\begin{bmatrix}\mathbf{0} &\calP_{\ba_1^{\perp}}\scrT\times_{q'\notin\{1,2\}}\ba^{q'} &\dots &\calP_{\ba_1^{\perp}}\scrT\times_{q'\notin\{1,p\}}\ba^{(q')}
		\\\calP_{\ba_2^{\perp}}(\scrT\times_{q'\notin\{1,2\}}\ba^{q'})^T &\mathbf{0} &\dots &\calP_{\ba_2^{\perp}}\scrT\times_{q'\notin\{2,p\}}\ba^{(q')}
		\\\vdots &\vdots &\ddots &\vdots
		\\\calP_{\ba_p^{\perp}}(\scrT\times_{q'\notin\{1,p\}}\ba^{q'})^T &\calP_{\ba_p^{\perp}}(\scrT\times_{q'\notin\{2,p\}}\ba^{q'})^T &\dots &\mathbf{0}
	\end{bmatrix}.$$
	Now for $(\bx_1\dots\bx_p)\in \calK,$ for any $k,\,l\in [p],$ $$\begin{aligned}\bx_{k}^T\calP_{\ba_k^{\perp}}\scrT\times_{q'\notin\{k,l\}}\ba^{(q')}\calP_{\ba_l^{\perp}}\bx_l=&\bx_{k}^T\scrT\times_{q\notin\{k,l\}}\bx_l
		\\=&\displaystyle\sum_{i \in S}\lambda_i\langle \bx_k,\,\bu^{(k)}_i\rangle \langle \bx_l,\,\bu^{(l)}_i\rangle\prod_{q'\notin\{k,l\}}\langle \ba^{(q')},\,\bu^{(q')}_i \rangle =0.
	\end{aligned}$$
	The first claim is now proved. For the rest, note similarly that for $(\bx_1,\dots,\bx_p)\in \calL,$ 
	$$\begin{aligned}
		&(\bx_1^T,\dots,\bx_p^T)D[\phi\circ G\circ \phi^{-1}](\bx_1^T,\dots,\bx_p^T)
		\\&=\dfrac{1}{\lambda}\sum_{k\neq l}\sum_{i\in S}\lambda_i\langle \bx_k,\bu^{(k)}_i\rangle \langle \bx_l,\bu^{(l)}_i\rangle \prod_{q'\notin\{k,l\}}\langle \ba_{q},\bu^{(q)}_i\rangle 
		\\&=\dfrac{1}{\lambda}\sum_{k\neq l}\sum_{i\in S}\abs*{\langle \bx_k,\bu^{(k)}_i\rangle}\abs*{ \langle \bx_l,\bu^{(l)}_i\rangle}\times\dfrac{\lambda_i\prod_q\langle \ba_q,\bu^{(q)}_i \rangle}{\abs*{\langle\ba_k,\bu^{(k)}_i  \rangle }\abs*{\langle\ba_l,\bu^{(l)}_i  \rangle }}
		\\&=\dfrac{1}{\lambda}\sum_{k\neq l}\sum_{i\in S}\abs*{\langle \bx_k,\bu^{(k)}_i\rangle}\abs*{ \langle \bx_l,\bu^{(l)}_i\rangle}\times \lambda
		\\&\ge \sum_{k\neq l}\bx_k^T\bx_l =\sum_{k,l}\bx_k^T\bx_l-\sum_{q=1}^p\bx_q^T\bx_q=p^2-p
		\\&>\sum_{q=1}^p\bx_q^T\bx_q=p,
	\end{aligned}$$
	where we use definition of $\calL$ in the second equality and \eqref{eq:values} in the third equality. Claim 2 then follows since $p>2.$
\end{proof}
Lemma \ref{le:hyperbolic} implies that the space spanned by the eigenvectors of $D[\phi\circ G\circ \phi^{-1}]$ with absolute eigenvalues less than 1, has dimension at most $(d-1)^{p}-(|S|-1)^p.$ Using theorem 4.17 from \cite{belkin2018eigenvectors} we obtain that, if $|S|>1,$ the local convergence manifold
$$\calL_{local}=\{\Tilde{\bx}(0):\lim_{t\to \infty}\bx_i(t)=\ba_i,\,\bx_i(t)\in U_i\qquad \forall t\in \mathbb{N}\}$$
has dimension strictly lower than that of $\calS^{(d-1)\otimes p}.$ On the other hand, it is immediate that the convergence manifold is full dimensional whenever $|S|=1.$

\noindent\textbf{Local to global:} Arguing along the lines of theorems 4.21-4.24 in \cite{belkin2018eigenvectors}, using the continuity and injectivity of $G,$ we can get a measure zero set $\calM$ such that for any $\tilde{\bx}=(\bx_1,\dots,\bx_p)\in \calS^{(d-1)\otimes p}/ \calM,$ we have $\eta>0$ and a critical point $(\ba_1,\dots,\ba_p)$ with $|S|>1$ such that  $(G_n(\tilde{\bx}))_q\in \calQ_{\ba_q},$
$$\max_{1\le q\le p}\norm*{\calP^{(q)}_{\bar{S}}(G_n(\tilde{\bx})_q)}\to 0,\,\,\rm{and}\,\,\norm*{G_n(\bx_1\otimes\dots\otimes \bx_p)-\ba_1\otimes\dots\otimes \ba_p}_F^2\ge \eta,$$
for all sufficiently large $n.$ 
To reduce notation, we use $\mathscr{U}_i,$ $\scrA$ and $\mathscr{X}$ to mean $(\bu_i^{(1)}\otimes \dots\otimes\bu_i^{(p)}),$ $(\ba_1\otimes \dots\otimes \ba_p)$ and $G_n(\bx_1\otimes\dots\otimes \bx_p)$ respectively. It can be checked by one application of $G$ that there exist $\eps>0$ and $i\neq j\in S$ such that
$$\dfrac{\langle \scrX,\,\scrU_i\rangle/\langle \scrA,\,\scrU_i\rangle}{\langle \scrX,\,\scrU_j\rangle/\langle \scrA,\,\scrU_j\rangle}>1+\eps.$$


We already have that with probability one, any starting point for the gradient iteration satisfies the claim above. We will now see that with each step  of the iteration, large inner products (between the estimate tensor and the hidden basis elements) become larger. Because of the norm constraint, this means that the estimate becomes more and more correlated with a particular basis element, eventually converging to it.

\begin{lemma}\label{le:poweriter} Suppose we have $\eps>0$ and $(\bx_1,\dots,\bx_p)$ satisfying  \\$\underset{i,j}{\max}\dfrac{\langle \scrX,\,\mathscr{U}_i\rangle/\langle \scrA,\,\mathscr{U}_i\rangle}{\langle \scrX,\,\mathscr{U}_j\rangle/\langle \scrA,\,\mathscr{U}_j\rangle}>1+\eps.$ Then,
	$$\underset{i,j}{\max}\dfrac{\langle G(\scrX),\mathscr{U}_i\rangle }{\langle G(\scrX),\mathscr{U}_j\rangle}\ge (1+\eps)^{p-2}\underset{i,j}{\max}\dfrac{\langle \scrX,\mathscr{U}_i\rangle}{\langle \scrX,\mathscr{U}_j\rangle }.$$
\end{lemma}
\begin{proof}[Proof of Lemma \ref{le:poweriter}]Let $i,j$ be the indices that maximize $\dfrac{\langle \scrX,\mathscr{U}_i\rangle /\langle \scrA,\,\mathscr{U}_i\rangle}{\langle \scrX,\mathscr{U}_j\rangle /\langle \scrA,\,\mathscr{U}_j\rangle}.$ By the definition of $G,$ we have
	$$\begin{aligned}\dfrac{\langle G(\scrX),\mathscr{U}_i\rangle }{\langle G(\scrX),\mathscr{U}_j\rangle}&=\dfrac{\lambda_i^p\left(\langle \scrX,\mathscr{U}_i\rangle \right)^{p-2}}{\lambda_j^p\left(\langle \scrX,\mathscr{U}_j\rangle \right)^{p-2}}\times \dfrac{\langle \scrX,\mathscr{U}_i\rangle }{\langle \scrX,\mathscr{U}_j\rangle}
		\\&\ge (1+\eps)^{p-2}\cdot\dfrac{\lambda_i^p\left(\langle \scrA,\mathscr{U}_i\rangle \right)^{p-2}}{\lambda_j^p\left(\langle \scrA,\mathscr{U}_j\rangle \right)^{p-2}}\cdot \dfrac{\langle \scrX,\mathscr{U}_i\rangle}{\langle \scrX,\mathscr{U}_j\rangle}
		\\&=(1+\eps)^{p-2}\cdot \dfrac{\lambda^{p}}{\lambda^p}\cdot \dfrac{\langle \scrX,\mathscr{U}_i\rangle }{\langle \scrX,\mathscr{U}_j\rangle},
	\end{aligned}$$
	where we use equation \eqref{eq:values} in the last step.
\end{proof}

By Lemma \ref{le:poweriter}, we obtain at least one $j\in S$ such that $\langle G_n(\scrX),\scrU_j\rangle\to 0.$ By the definition of $G$ it can be checked that this in turn implies that \\$\max_{1\le q\le p}|\langle (G_n(\scrX))_q, \bU^{(q)}_j\rangle|\to 0$ for some $j\in S.$

Repeated use of Lemma \ref{le:poweriter} gives us that for any starting point $\scrX(0)$ in a probability one set, we have a critical point $(\ba_1,\dots,\ba_p)$ with $S=\{i\in[d]:\langle \scrA,\mathscr{U}_i\rangle \neq 0\}$ satisfying $|S|>1,$ such that $\norm*{\calP^{(q)}_{\bar{S}}(G_n(\scrX))_q}\to 0$ and there is at least one $j\in S$ for which $\langle G_n(\scrX),\,\mathscr{U}_j\rangle \to 0$ as $n\to \infty.$ We can now repeat the entire argument to get a decreasing sequence $S=S_0\supset S_1\supset \dots S_k$ such that $|S_k|=|S|-k$ and $\langle G_n(\scrX),\mathscr{U}_i\rangle\to 0 $ for all $i\notin S_k.$ Therefore $G_n(\scrX)\to (\bu^{(1)}_i,\dots,\bu^{(p)}_i)$ for some $i$ with $\lambda_i>0.$ This finishes the proof of theorem \ref{th:comp}.
$\qed$

\bigskip

\begin{proof}[Proof of Lemma \ref{le:specnorm}]
	Observe that, or any $\eps>0$, $\|\tilde{\scrT}-\scrT\|\le c_\eps\lambda_{\min}$ for the constant $c_\eps$ in Theorem \ref{th:ortho-perturb}, provided $d_{\min}>C_\eps$ (a large enough constant depending on $\eps$), so that the conditions of Theorem \ref{th:ortho-perturb} are satisfied.
	
	Then by \eqref{eq:secordpert}, there is a vector $\boldsymbol{\gamma}^{(q)}\in\{+1,-1\}^{d_{\min}}$ such that $$\bV^{(q)}_k=[\bu_1^{(q)}\,\dots\,\bu_{k}^{(q)}]\,\,\text{ and}\,\,
	\tilde{\bV}^{(q)}_k=[\gamma_1^{(q)}\tilde{\bu}_{\pi(1)}^{(q)}\,\dots\,\gamma_{k}^{(q)}\tilde{\bu}_{\pi(k)}^{(q)}]$$
	satisfy
	\begin{equation*}\label{eq:addMR}
		\bV^{(q)}_k-\tilde{\bV}^{(q)}_k=\bM^{(q)}_k\bD_k+\bR^{(q)}_k
	\end{equation*}
	where $\bM^{(q)}_k$ is the matrix defined in Assumption {\bf A1} and $\bD_k={\rm diag}(\lambda_1^{-1},\dots,\lambda_k^{-1})$. The remainder $\bR^{(q)}_k$ has columns satisfying
	\begin{align*}\label{eq:Rqbound}
		\|\bR^{(q)}_k\be_i\|\le& 2\left(2+\dfrac{\|\tilde{\scrT}-\scrT\|}{\lambda_i}\right)
		\left(\dfrac{(1+\eps)\|\tilde{\scrT}-\scrT\|}{\lambda_i}\right)^{p-1}\\
		\le & C\left(\dfrac{\|\tilde{\scrT}-\scrT\|}{\lambda_i}\right)^{p-1}\numberthis
	\end{align*}
	for $1\le i\le k$, where we use \eqref{eq:secordpert} in the first line and Assumption {\bf A2} in the second line. We thus have
	\begin{align*}
		\|\bV^{(q)}_k-\tilde{\bV}^{(q)}_k\|\le &\|\bM^{(q)}_k\bD_k+\bR^{(q)}_k\|\\
		\le &\|\bM^{(q)}_k\|\|\bD_k\|+\|\bR^{(q)}_k\|\\
		\le & 
		\dfrac{C\|\tilde{\scrT}-\scrT\|}{\lambda_{\min,k}}+\sqrt{d_{\min}}\max_{1\le i\le k}\|\bR^{(q)}\be_i\|\\
		\le & \dfrac{C\|\tilde{\scrT}-\scrT\|}{\lambda_{\min,k}}+\sqrt{d_{\min}}\cdot C\left(\dfrac{\|\tilde{\scrT}-\scrT\|}{\lambda_{\min,k}}\right)^{p-1}\\
		\le &\dfrac{C\|\tilde{\scrT}-\scrT\|}{\lambda_{\min,k}}	
	\end{align*}
	where we used \eqref{eq:Rqbound} in the third inequality. Finally, if $d_{\min}<C_{\eps}$, the conclusion follows directly using the column-wise error bounds from Theorems \ref{th:ortho-perturb} and \ref{th:odeco-weyl}.
\end{proof}

\end{document}